\documentclass[12pt]{amsart}
\usepackage[latin1]{inputenc}
\usepackage{amsxtra}
\usepackage{amsmath}
\usepackage{amscd}
\usepackage{amssymb}
\usepackage{amsfonts}
\usepackage[all]{xy}
\usepackage{mathrsfs}
\usepackage{amsthm}
\usepackage{color}
\usepackage{upgreek}
\usepackage{verbatim}
\usepackage{enumerate}
\usepackage{latexsym}
\usepackage{graphicx}
\usepackage{todonotes}
\usepackage{setspace}
\usepackage{hyperref}
\usepackage{stmaryrd}
\usepackage{nicefrac}
\usepackage{euscript}
\usepackage{mathdots}
\usetikzlibrary{decorations.pathmorphing}

 \definecolor{darkgreen}{HTML}{336633}
 \definecolor{darkred}{HTML}{993333}
\makeatletter
\newcommand{\arxiv}[1]{\href{http://arxiv.org/abs/#1}{\tt
    arXiv:\nolinkurl{#1}}}

\hypersetup{colorlinks,linkcolor=blue,urlcolor=cyan,citecolor=blue}

\theoremstyle{plain}
\newtheorem{thm}{Theorem}[section]
\newtheorem{lem}[thm]{Lemma}
\newtheorem{prop}[thm]{Proposition}

\newtheorem{cor}[thm]{Corollary}

\theoremstyle{definition}
\newtheorem{df}[thm]{Definition}

\theoremstyle{remark}
\newtheorem{rem}[thm]{Remark}
\newtheorem{ex}[thm]{Example}

\xyoption{all}
\def\bbA{\mathbb{A}}

\def\bbC{K}

\def\bbJ{\mathbb{J}}

\def\bbN{\mathbb{N}}

\def\bbZ{\mathbb{Z}}

\def\scrD{\mathscr{D}}

\def\scrK{\mathscr{K}}

\def\scrN{\mathscr{N}}

\def\scrR{\mathscr{R}}

\def\scrT{\mathscr{T}}

\def\scrV{\mathscr{V}}
\def\scrW{\mathscr{W}}

\def\scrZ{\mathscr{Z}}

\def\calA{\mathcal{A}}

\def\calC{\mathcal{C}}
\def\calD{U}

\def\calU{\mathcal{U}}

\def\ML{\mathcal{P}}
\def\calN{V}
\def\calV{\mathcal{V}}

\def\bfm{\mathbf{m}}

\def\bfr{\mathbf{r}}

\def\bfv{\mathbf{v}}

\def\TPL{\mathcal{PG}}

\def\rank{\mathrm{rank}}

\def\L{{\rm L}}
\def\la{\lambda}
\def\ov{\overline}
\def\diag{\mathrm{diag}}

\renewcommand{\hom}{\operatorname{hom}\nolimits}
\newcommand{\Hom}{\operatorname{Hom}\nolimits}
\renewcommand{\mod}{\operatorname{mod}\nolimits}
\newcommand{\Ext}{\operatorname{Ext}\nolimits}
\newcommand{\ext}{\operatorname{ext}\nolimits}

\newcommand{\GL}{\operatorname{GL}\nolimits}

\newcommand{\QLB}{\operatorname{LB}^q\nolimits}

\title[Varieties of modules over the quantum plane]{Varieties of modules over the quantum plane}
\author[Xinhong Chen]{Xinhong Chen}
\address{Department of Mathematics, Southwest Jiaotong University, Chengdu 610031, P.R.China}
\email{chenxinhong@swjtu.edu.cn}

\author[Ming Lu]{Ming Lu}
\address{Department of Mathematics, Sichuan University, Chengdu 610064, P.R.China}
\email{luming@scu.edu.cn}


\begin{document}
\begin{abstract}
The quantum plane is the non-commutative polynomial algebra in variables $x$ and $y$ with $xy=qyx$.
In this paper, we study the module variety of $n$-dimensional modules over the quantum plane, and provide an explicit description of its irreducible components and their dimensions. We also describe the irreducible components and their dimensions of the GIT quotient of the module variety with respect to the conjugation action of $\GL_n$.
\end{abstract}

\numberwithin{equation}{section}

\maketitle

 \setcounter{tocdepth}{1}
 \tableofcontents

\section{Introduction}

\subsection{Background}
Let $K$ be an algebraically closed field of characteristic zero.
When considering representation theory of a (finitely generated) associative $K$-algebra $A$, one would like to classify its finite-dimensional modules (up to isomorphisms). However, by strong consensus, it is an unattainable goal in general.
A more promising alternative is to classify the finite-dimensional modules generically, that is, to study the {\em module variety} $\mod^n(A)$ of $n$-dimensional modules \cite{CB-S}.

It is a basic problem to study the irreducible components of $\mod^n(A)$ and their dimensions. Moreover, the general linear group $\GL_n$ acts on $\mod^n(A)$ by conjugation, and its orbits correspond to the isomorphism classes of $n$-dimensional modules. It leads us to study the GIT quotient, especially its irreducible components and their dimensions. 
The varieties of modules over algebras are studied and applied widely, see \cite{Ri79,Lu91,CB-S, Pr03, Boz16,GH18} and the references therein.


\subsection{Goal}
In this paper, we consider the module varieties over quantum plane, see e.g. \cite[Chapter I.2.1]{BG02}. A quantum plane is a non-commutative algebra generated by $x,y$ with defining relation $xy=qyx$ for some fixed nonzero scalar $q$. 
When $q=1$,
this module variety is the commuting variety studied by Motzkin and Taussky \cite{MT55} and Gerstenhaber \cite{Ger61}; When $q=-1$, it is the anti-commuting variety studied by the first author and Wang \cite{CW18}, whose irreducible components and their dimensions are described explicitly.

If $q$ is a root of unity, let $\ell\in\bbZ_{>0}$ be the smallest number such that $q^\ell=1$, otherwise let $\ell:=\infty$.
The {\em $q$-commuting variety} is defined to be
\begin{equation*}
\scrK_{q,n} =\{(A, B) \in M_{n\times n}(\bbC) \times M_{n\times n}(\bbC) \mid AB=qBA \}.
\end{equation*}

Following \cite{CB-S}, for any $n\geq1$, the {module variety} $\mod^n(\bbA_q^2)$ is the set of $\bbA_q^2$-module structures on $\bbC^n$, or equivalently the set of $\bbC$-algebra homomorphism from $\bbA_q^2$ to $M_{n\times n}(\bbC)$. Now such a homomorphism is determined by the values on $x,y$.
Then we have
\begin{equation*}
\mod^n(\bbA_q^2)\cong\scrK_{q,n}.
\end{equation*}

The aim of this paper is to understand the geometric properties of $\scrK_{q,n}$ and its GIT quotient, especially the irreducible components and their dimensions. In order to achieve this, beside the linear algebra techniques used in \cite{CW18}, we in addition apply the homological algebra techniques and decomposition properties of modules following \cite{CB-S}.

\subsection{The main result}
Let $\ML_{q,n}$ be the indexing set defined in \eqref{ML}. For each $(\bfm,\bfr)$ in $\ML_{q,n}$, a closed subvariety  $\scrK_{q,(\bfm,\bfr)}$ (see \eqref{eqn: def of irr}) of $\scrK_{q,n}$ is introduced. Then the first main result of this paper is the following:


\begin{thm}[Theorem \ref{thm: irreducible comp}, Theorem \ref{prop: dimension of Cl}, Proposition \ref{prop: irreducible l infty}]
\label{main theorem}
For each $(\bfm,\bfr)\in\ML_{q,n}$, the variety $\scrK_{q,n}$ has one irreducible component $\scrK_{q,(\bfm,\bfr)}$ so that
$\scrK_{q,n} =\bigcup_{(\bfm,\bfr)\in \ML_{q,n}} \scrK_{q,(\bfm,\bfr)}$, and $\dim \scrK_{q,(\bfm,\bfr)}=n^2+m_\ell$ if $\ell<\infty$; $\dim \scrK_{q,(\bfm,\bfr)}=n^2$ if $\ell=\infty$.
\end{thm}




Let $G=\GL_n$. We consider the conjugation action of $G$ on $\scrK_{q,n}$, and the GIT quotient $\scrK_{q,n}//G$. In order to describe the irreducible components of $\scrK_{q,n}//G$, another indexing set $\TPL_{q,n}$ is introduced in \eqref{TPL}, and the variety $\scrZ_{p,m,r}//G$ is defined for $(p,m,r)\in\TPL_{q,n}$ in \eqref{def:Zpmr}.

\begin{thm}[Proposition \ref{prop: irr 2}, Proposition \ref{prop:pure dimension}, Lemma \ref{lem: l=0 irr}, Lemma \ref{lem:pure dimension}]
\label{main thm 2}
For each $(p,m,r)\in\TPL_{q,n}$, the GIT quotient $\scrK_{q,n}//G$ has one irreducible component $\scrZ_{q,(p,m,r)}//G$ so that
$\scrK_{q,n} //G=\bigcup_{(p,m,r)\in \TPL_{q,n}} \scrZ_{q,(p,m,r)}//G$. Moreover, we have $\dim \scrZ_{q,(p,m,r)}//G=n+(2-\ell)p$ if $\ell<\infty$; $\dim \scrZ_{q,(p,m,r)}//G=n$ if $\ell=\infty$.
\end{thm}

The quantum plane can be realised as a quiver algebra, with one vertex and two loops. It is natural to study modules over more general quiver algebras with loops.

Meanwhile, the $r$-tuple commuting varieties are studied in \cite{Ger61,KN87}, so it will be interesting to study the geometric properties of $r$-tuple $q$-commuting varieties, i.e., the module varieties over quantum affine $r$-spaces.




\subsection{The organization}
The paper is organized as follows. We describe the concrete form of matrices $q$-commute with a matrix in Jordan normal form in Section~\ref{sec: pre}. In Section~\ref{sec: q-commuting var}, the subvarieties $\scrD_{q,i}$, $\scrN_{q,j}$ of $\scrK_{q,n}$ are defined, thus the subvariety $\scrK_{q,(\bfm,\bfr)}$ of $\scrK_{q,n}$ is defined for each $(\bfm,\bfr)$ in the indexing set $\ML_{q,n}$.  

We show $\scrK_{q,(\bfm,\bfr)}$ is an irreducible component of $\scrK_{q,n}$ for each $(\bfm,\bfr)\in \ML_{q,n}$ in Section~\ref{subsec: irreducible}. 
By a reduction via the ``direct sum" property and the homological properties described in Lemma \ref{lem:extension0}, we only need to prove that $\scrD_{q,i}$ and $\scrN_{q,j}$ are irreducible components, see Theorem \ref{prop: irreducible components for smaller n}.

In Section~\ref{sec: further prel}, we  introduce the notion of $q$-chains when $q$ is a root of unity, in order to figure out which component $\scrK_{q,(\bfm,\bfr)}$ contains a given pair $(A,B) \in \scrK_{q,n}$, where $A$ has eigenvalues $aq^{-i}$ with a fixed  $a\in \bbC^*$.

It is established in Section~\ref{sec: cover} that  the irreducible components of $\scrK_{q,n}$ are all $\scrK_{q,(\bfm,\bfr)}$ when $q$ is a root of unity. 
We determine the component $\scrK_{q,(\bfm,\bfr)}$ which contains a given pair $(A,B) \in \scrK_{q,n}$, where either $A$ has eigenvalues $aq^{-i}$ for a fixed $a\in \bbC^*$( see Proposition \ref{Prop:QCJor} and Proposition \ref{prop:block a,aq^{-(l-1)}}), or $A$ is nilpotent (see Proposition~\ref{prop:Jordan block 0}).

Section~\ref{sec: dim} is devoted to compute the dimension of each irreducible component $\scrK_{q,(\bfm,\bfr)}$. In Theorem \ref{prop: dimension of Cl}, we prove that $\dim \scrK_{q,(\bfm,\bfr)}=n^2+m_\ell$. 

In Section~\ref{sec:GIT}, we describe the irreducible components of $\scrK_{q,n}//\GL_n$, and their dimensions.

In section~\ref{section:infinite}, we summarize the results for the case $\ell=\infty$.

 \vspace{.4cm}
 {\bf Acknowledgment.}
This paper answers a question posted by Weiqiang Wang,
we would like to thank him for numerous stimulating discussions. We are grateful to University of Virginia for hospitality during our visit when this project was initiated. The authors are partially supported by NSFC grant No.~11601441. 

\section{Preliminaries}
\label{sec: pre}
Throughout this paper, $K$ will be an algebraically closed filed of characteristic zero, and $q\in\bbC^*$. Denote by $M_{n\times n}(\bbC)$ the space of $n\times n$ matrices over $\bbC$.

We denote by $\bbN$ and $\bbZ$ the sets of nonnegative integers and integers respectively. Let $\ell\in\bbZ_{>0}$ be the smallest number such that $q^\ell=1$ if $q$ is a root of unity; and $\ell:=\infty$ otherwise.

\subsection{$q$-commuting varieties}

The {\em $q$-commuting variety} is
\begin{equation}
\scrK_{q,n} =\{(A, B) \in M_{n\times n}(\bbC) \times M_{n\times n}(\bbC) \mid AB=qBA \}.
\end{equation}

For any $m\leq n$, $\scrK_{q,m}$ is a subvariety of $\scrK_{q,n}$ by viewing $(A,B)$ in $\scrK_{q,m}$ as $(\diag(A,0),\diag(B,0))$ in $\scrK_{q,n}$.
In this way, define
\begin{equation}
{\scrK}_{q,\infty}:= \lim\limits_{\longleftarrow}(\scrK_{q,n}).
\end{equation}
\label{def:RqA}
For any $A\in M_{n\times  n}(\bbC)$, define
\begin{align}
\scrR_{q,A}:=\{B\in M_{n\times n}(\bbC)\mid AB=qBA\}.
\end{align}

By definition,
\begin{equation}
\scrK_{q^{-1},n} =\{(A, B) \in M_{n\times n}(\bbC) \times M_{n\times n}(\bbC) \mid AB=q^{-1}BA \},
\end{equation}
so there exists a natural morphism of varieties
\begin{align}
\label{def: theta}
\theta: \scrK_{q,n} \rightarrow \scrK_{q^{-1},n}
\end{align}
by mapping $(A,B)$ to $(B,A)$. Obviously, $\theta$ is an isomorphism of varieties. 

 For any $A,B\in M_{n\times n}(\bbC)$, 
 we say $B$ \emph{$q$-commutes with $A$} if $AB=qBA$.

In the remainder of this section, we describe all matrices which $q$-commute with a matrix in Jordan normal form. The notations introduced here are based on \cite{CW18}.

\subsection{$q$-layered matrix}
Consider $m\times n$ matrices over $\bbC$ of the form:
\begin{eqnarray}
\label{eq:block a1}
\L^q(m,n,\vec{v})=\left [\begin{array}{cccccccccc}
 0      & \cdots  & 0      & b_1     & b_2        &\cdots     &  b_{m-1}             & b_m                 \\
 0      & \cdots  & 0      & 0       & qb_1  &\cdots          & b_{m-2}             & qb_{m-1}           \\
 \vdots &         & \vdots & \vdots  & \vdots           & \ddots    & \vdots               & \vdots             \\
 0      & \cdots  & 0      & 0       &  0         &\cdots     & q^{m-2}b_1          &q^{m-2}b_2       \\
 0      & \cdots  & 0      & 0       & 0          &\cdots     & 0                    &q^{m-1}b_1
\end{array} \right], \ \text{if}\ m\leq n,
\end{eqnarray}
or
\begin{eqnarray}
\label{eq:block a2}
\L^q(m,n,\vec{v})=\left [\begin{array}{cccccccccc}
  b_1     & b_2      &    \cdots       & b_{n-1}           &b_n                \\
   0      &qb_1      & \cdots          &qb_{n-2}           &qb_{n-1}            \\
   \vdots       &      \vdots    &  \ddots         &      \vdots             &       \vdots             \\
   0      &    0     &    \cdots             & q^{n-2} b_1    &    q^{n-2} b_2       \\
   0      &     0    &   \cdots        &     0             &   q^{n-1} b_1    \\
   0      &    0     &   \cdots        &   0               &  0                  \\
 \vdots   &\vdots    &                 &\vdots             &\vdots               \\
  0       &    0     &    \cdots       &    0              &  0
 \end{array} \right], \ \text{if}\ m\geq n,
\end{eqnarray}
where $\vec{v}=(b_1,\ldots,b_{\min(m,n)})$. 
We call matrices of the form $\L^q(m,n,\vec{v})$ {\em $q$-layered matrices}.

Let $J_m(\alpha)$ be the $m\times m$ Jordan block with eigenvalue $\alpha$. The following lemma can be proved by a direct computation and we omit the details.
\begin{lem}
   \label{Jordan q-layer}
For $\alpha_1, \alpha_2 \in \bbC$, a matrix $B$ satisfies $J_{m}(\alpha_1) B =qB J_{n}(\alpha_2)$ if and only if
\begin{eqnarray*}
B=\begin{cases}
  0                                                    & \   \text{if}\  \alpha_1\neq q\alpha_2 \\
\L^q(m,n,\vec{v})                                        & \  \text{if}\   \alpha_1=q\alpha_2,
\end{cases}
\end{eqnarray*}
for some vector $\vec{v}$. 
\end{lem}

\begin{prop}
\label{QCmatrix}
Let $A=\diag(J_{m_1}(\alpha_1),J_{m_2}(\alpha_2),...,J_{m_r}(\alpha_r))$. Then a matrix $B$ satisfies $AB=qBA$ if and only if $B=(B_{ij})$ is an $r\times r$ block matrix with its $m_i\times m_j$-submatrix
\begin{eqnarray*}
B_{ij}=\begin{cases}
0                                                       &   \text{if}\ \alpha_i\neq q\alpha_j \\
\L^q( m_i,m_j,\vec{v}_{ij})                               &   \text{if}\    \alpha_i=q\alpha_j,
\end{cases}
\end{eqnarray*}
for arbitrary $\min (m_i,m_j)$-tuple vectors $\vec{v}_{ij}$ for all $i,j$.
\end{prop}

\begin{proof}
The $(i,j)$-th block of the equation $AB-qBA=0$ reads $ J_{m_i}(\alpha_i)B_{ij}-qB_{ij}J_{m_j}(\alpha_j)=0$. The proposition now follows from Lemma~\ref{Jordan q-layer}.
\end{proof}
\subsection{q-layered block matrices}

Let
\begin{align}
\label{eqn: J}
&\mathbb{J}_{s,n}(0)=
{
\begin{bmatrix}
0      &  I_n  & 0     & \ldots  &  0      & 0              \\
0      &  0    &I_n    & \ldots  &  0      & 0              \\
\vdots &\vdots &\vdots &         & \vdots  &  \vdots   \\
0      &  0    &0      & \ldots  &  0      & I_n             \\
0      &  0    &0      & \ldots  &  0      & 0              \\
 \end{bmatrix},
 }
\end{align}
be an $sn\times sn$ matrix in $s\times s$ blocks,
where $I_n$ is the $n\times n$ identity matrix. Notice that $\mathbb{J}_{1,n}(0)$ is the $n\times n$ zero matrix. For any $\alpha \in \bbC$, let
\[
\mathbb{J}_{s,n}(\alpha) := \alpha I_{sn} +\mathbb{J}_{s,n}(0),
\]
and
\[
J_s^n(\alpha) :=\diag(\underbrace{J_s(\alpha),\ldots,J_s(\alpha)}_{n}).
\]
Notice $J_s^n(\alpha) = \alpha I_{sn} + J_s^n(0)$.

\begin{lem}[\cite{CW18}]
  \label{lem:sJordan}
The matrix $J_s^n(\alpha)$ is similar to the  matrix $\mathbb{J}_{s,n}(\alpha)$.
\end{lem}


%
Consider the following $s\times t$ block matrices
\begin{align}\label{eq:block al1}
\QLB(s,t,\vec{V})=
\begin{bmatrix}
B_1      &  B_2     & \cdots  & B_{t-1}    & B_t             \\
0        & qB_1     & \cdots  &qB_{t-2}    & qB_{t-1}           \\
\vdots   &\vdots    &         &\vdots      &  \vdots   \\
0        &  0       & \cdots  &q^{t-2}B_1  &  q^{t-2}B_2          \\
0        &  0       & \cdots  & 0          &  q^{t-1}B_1      \\
0        &  0       & \cdots  &0           &  0      \\
\vdots   & \vdots   &         &\vdots           & \vdots     \\
0  & 0   &     \cdots    &0          & 0
 \end{bmatrix}, \text{ if}\ s\geq t;
\end{align}
or
\begin{align}\label{eq:block al2}
\QLB(s,t,\vec{V})=
\begin{bmatrix}
0      &\cdots   & 0     & B_1      &  B_2       & \cdots   & B_{s-1}     & B_s             \\
0      & \cdots  & 0     & 0        & qB_1       & \cdots   & qB_{s-2}    & qB_{s-1}           \\
\vdots &                 &  \vdots  & \vdots     &  \vdots&        &\vdots       &  \vdots   \\
0      &\cdots   &0      & 0        & 0          &  \cdots  & q^{s-2}B_{1}& q^{s-2}B_{2} \\
0      & \cdots  &0      & 0        & 0          &  \cdots  &  0          & q^{s-1}B_1\\
 \end{bmatrix}, \text{ if}\ s\leq t,
\end{align}
where $\vec{V}=(B_1,\dots,B_{\min(s,t)})$, and $B_{i}$'s are matrices of the same size.
We call the block matrices of the form $\QLB(s,t,\vec{V})$ \emph{q-layered block matrices}.

The following lemma and proposition are natural generalizations of Lemma~ \ref{Jordan q-layer} and Proposition~\ref{QCmatrix}. The proofs are again by a direct computation and will also be skipped.

\begin{lem}
 \label{lem:QLB}
For $\alpha_1, \alpha_2 \in \bbC$, a matrix $B$ satisfies  $\mathbb{J}_{s,n_s}(\alpha_1)B=qB\mathbb{J}_{t,n_t}(\alpha_2)$ if and only if
\begin{eqnarray*}
B=\begin{cases}
  0                                                       & \   \text{if}\  \alpha_1\neq q\alpha_2 \\
   \QLB(s,t,\vec{V})                                        & \  \text{if}\   \alpha_1=q\alpha_2,
\end{cases}
\end{eqnarray*}
for $\vec{V}=(B_1,\dots,B_{\min(s,t)})$ with $n_s\times n_t$ matrices $B_i$ for all $i$. 
\end{lem}

\begin{prop}
\label{QLBmatrix}
Let $A=\diag(\mathbb{J}_{1,n_1}(\alpha_1), \mathbb{J}_{2,n_2}(\alpha_2),..., \mathbb{J}_{r,n_r}(\alpha_r))$. Then a matrix $B$ satisfies $AB=qBA$ if and only if $B=(B_{ij})$ is an $r\times r$ block matrix with its $in_i\times jn_j$-submatrix
\begin{eqnarray*}
B_{ij}=\begin{cases}
0                                                       &   \text{if}\ \alpha_i\neq q\alpha_j \\
\QLB( i,j,\vec{V}_{ij})                               &   \text{if}\    \alpha_i=q\alpha_j;
\end{cases}
\end{eqnarray*}
where $\QLB( i,j,\vec{V}_{ij})$ is an $i\times j$ block matrix (whose block size is $n_i\times n_j$) for every $i$ and $j$. 
\end{prop}

\section{$q$-commuting varieties}
\label{sec: q-commuting var}

We always assume that $\ell<\infty$ in the following except Section \ref{section:infinite}.

\subsection{Definition of $\scrD_{q,i}$ and $\scrN_{q,i}$}
\label{subsec: def DN}
Recall $G=\GL_n$. Then there is a conjugation action of $\GL_n$ on $\scrK_{q,n}$, i.e.,
$g\cdot (A,B)=(gAg^{-1},gBg^{-1})$ for any $g\in\GL_n$, $(A,B)\in \scrK_{q,n}$.

For any subset $\scrV\subseteq \scrK_{q,n}$, we denote by
$\ov{\scrV}$ its Zariski closure in $\scrK_{q,n}$. 

\begin{df}
\label{df: D}
For each $1\leq i\leq \ell$, the (constructible) subset $\scrD^\circ_{q,i}$ 
is defined to be the union of $\GL_i$-orbits of all $(A,B) \in \scrK_{q,i}$, where
$A$ is a diagonal matrix of the form
\begin{equation}\label{eq:A1}
A=\diag(a,aq^{-1},\ldots,aq^{-(i-1)})\, ,\  \text{for}\   a\in \bbC^*.
\end{equation}
The variety $\scrD_{q,i}$ is defined to be the Zariski closure of $\scrD^\circ_{q,i}$ in $\scrK_{q,i}$.
\end{df}
\begin{df}
\label{df: N}
For each $1\leq i<\ell$, the (constructible) subset $\scrN^\circ_{q,i}$ is defined to be the union of $\GL_i$-orbits  of $(J_i(0),B) \in \scrK_{q,i}$. The variety $\scrN_{q,i}$ is defined to be the Zariski closure of  $\scrN^\circ_{q,i}$ in $\scrK_{q,i}$.
\end{df}

Note that $\scrD^\circ_{q,i}$ is an open subset of $\scrD_{q,i}$ for $1\leq i\leq \ell$; and $\scrN^\circ_{q,i}$ is an open subset of $\scrN_{q,i}$ for $1\leq i<\ell$.


\begin{df}
\label{rem: calD calN}

(i) For any $1\leq i <\ell$, denote by $U_{q,i}$ the union of $\GL_i$-orbits of all $(A,B) \in \scrK_{q,i}$, where
\begin{eqnarray*}
A=\diag(a,aq^{-1},\ldots,aq^{-(i-1)})\, ,\  \text{for}\   a\in \bbC^*,
\end{eqnarray*}
and $B=\left[ \begin{array}{ccccccc} 0&b_1&0& \cdots& 0\\ &0& b_2&\cdots &0\\ &&\ddots&\ddots&\vdots& \\ &&&\ddots&b_{i-1}\\&&&&0\end{array} \right]$ for $b_j\in\bbC^*$. 

(ii) Denote by $U_{q,\ell}$ the union of $\GL_i$-orbits of all $(A,B)\in\scrK_{q,\ell}$, where $A=\diag(a,aq^{-1},\ldots,aq^{-(\ell-1)})$ for $a\in\bbC^*$, and
\begin{align}
\label{eqn: form B}
B=\begin{bmatrix}
0   & b_1    &        &              &          \\
    &  0     & b_2    &              &        \\
    &        &\ddots   &\ddots        &        \\
    &        &        &   0          & b_{\ell-1}  \\
b_\ell &        &        &              &  0
\end{bmatrix}
\end{align}
with $b_j\in\bbC^*$ for any $1\leq j\leq \ell$.

(iii) For any $1\leq i<\ell$, denote by $V_{q,i}$ the union of $\GL_i$-orbits of all $(J_{i}(0),B) \in \scrK_{q,i}$, where
$B=\L^q(i,i,(b_1,\dots,b_i))$ with $b_j\in\bbC^*$ for any $1\leq j\leq i$. 
\end{df}

\begin{lem}
We have
\begin{itemize}
\item[(i)] $\scrD_{q,i}=\ov{U_{q,i}}$ for $1\leq i\leq\ell$;
\item[(ii)] $\scrN_{q,i}=\ov{V_{q,i}}$ for $1\leq i< \ell$.
\end{itemize}
\end{lem}

\begin{proof}
We only prove (i) for the case $i<\ell$. For any $(A,B)\in\scrD_{q,i}^\circ$, by definition, there exists a $g\in\GL_i$ such that $gAg^{-1}=\diag(a,aq^{-1},\ldots,aq^{-(\ell-1)})$ for some $a\in K^*$. Proposition \ref{QCmatrix} shows that $gBg^{-1}=\left[ \begin{array}{ccccccc} 0&b_1&0& \cdots& 0\\ &0& b_2&\cdots &0\\ &&\ddots&\ddots&\vdots& \\ &&&\ddots&b_{i-1}\\&&&&0\end{array} \right]$ with $b_j\in\bbC$.
It follows that $U_{q,i}$ is a dense open subset of $\scrD_{q,i}^\circ$. So $\scrD_{q,i}=\ov{U_{q,i}}$ by definition.
\end{proof}

\subsection{Quantum plane}
Let $\bbA_q^2$ be the quantum plane. 
 Following \cite{CB-S}, for any $n\geq1$, the module variety $\mod^n(\bbA_q^2)$ is the set of $\bbA_q^2$-module structures on $\bbC^n$, or equivalently the set of $\bbC$-algebra homomorphism from $\bbA_q^2$ to $M_{n\times n}(\bbC)$. Now such a homomorphism is determined by the values on $x,y$. 
Throughout this paper, we always identify $\mod^n(\bbA_q^2)$ with $\scrK_{q,n}$. In this way, an $n$-dimensional $\bbA_q^2$-module is denoted by $M(A,B)$ (or just $(A,B)$ if there is no confusions)
for some $(A,B)\in \scrK_{q,n}$.


Inspired by the definition of direct sum of modules, we define
\begin{align}
&(A_1,B_1)\oplus(A_2,B_2)\\\notag
:=&(\diag (A_1,A_2), \diag (B_1,B_2))\in\scrK_{q,i+j},\ \forall (A_1,B_1)\in \scrK_{q,i},(A_2,B_2)\in \scrK_{q,j}.
\end{align}

Similarly, we can define the Hom-space and Ext-space for any $(A_1,B_1)$, $(A_2,B_2)$ in $\scrK_{q,\infty}$, which are denoted by
$$\Hom_{\bbA_q^2}((A_1,B_1),(A_2,B_2)), \text{ and }\Ext^1_{\bbA_q^2}((A_1,B_1),(A_2,B_2))$$
respectively.

Let $\scrV\subseteq \scrK_{q,i},\scrW\subseteq \scrK_{q,j}$ be two subsets. We give the following notations.
\begin{align}
\scrV\oplus \scrW:=&\{(A_1,B_1)\oplus(A_2,B_2) \in \scrK_{q,i+j}\mid (A_1,B_1)\in \scrV,(A_2,B_2)\in \scrW\};\\
{\hom}_{\bbA_q^2}(\scrV,\scrW):=&\min\{\dim_\bbC \Hom_{\bbA_q^2}((A_1,B_1),(A_2,B_2))\mid (A_1,B_1)\in \scrV,(A_2,B_2)\in\scrW\};\\
\ext^1_{\bbA_q^2}(\scrV,\scrW):=&\min\{\dim_\bbC \Ext^1_{\bbA_q^2}((A_1,B_1),(A_2,B_2))\mid (A_1,B_1)\in \scrV,(A_2,B_2)\in\scrW\}.
\end{align}

\begin{df}[$q$-equivalent]
\label{def: l relation}
Two nonzero numbers $a_1$ and $a_2$ are said to be $q$-equivalent if
$a_1=a_2q^m$ for some $m\in\bbZ$.
A multiset of numbers $a_1,\ldots,a_s$ is not $q$-equivalent if $a_i$ and $a_j$ are not $q$-equivalent for any $i\neq j$.
\end{df}
Here a {\em multiset} is a generalization of the concept of a set that, unlike a set, allows multiple instances for each of its elements.

\begin{lem}
\label{lem:extension0}
For any $\scrV,\scrW\in \{ \scrD_{q,1}\dots,\scrD_{q,\ell-1}, \scrD_{q,\ell},\scrN_{q,1} ,\dots, \scrN_{q,\ell-1} \}$,
we have
\begin{align*}
&\hom_{\bbA_q^2}(\scrV,\scrW)=0;\\
&\ext^1_{\bbA_q^2}(\scrV,\scrW)=0.
\end{align*}
\end{lem}

\begin{proof}
The proof is divided into the following four cases.

Case (i): \underline{$\scrV=\scrD_{q,i},\scrW=\scrD_{q,j}$}. By definition, we can choose $L=(A_1,B_1)\in \scrD_{q,i}$, and $N=(A_2,B_2)\in\scrD_{q,j}$ such that
\begin{itemize}
\item $A_1=\diag(a,aq^{-1},\ldots,aq^{-(i-1)})$, and $A_2=\diag(c,cq^{-1},\ldots,cq^{-(j-1)})$,
\item $a,c\in\bbC^*$ are not $q$-equivalent.
\end{itemize}
For any morphism in $\Hom_{\bbA_q^2}(N,L)$, it is represented by a $j\times i$ matrix $F$ such $FA_1=A_2F$, $FB_1=B_2F$.
It is routine to check that $F=0$. So $\Hom_{\bbA_q^2}(N,L)=0$, which shows that $\hom_{\bbA_q^2}(\scrD_{q,i},\scrD_{q,j})=0$.

On the other hand, for any short exact sequence
\begin{equation}\label{eq:LMN}
0\rightarrow N\rightarrow M\rightarrow L\rightarrow0,
\end{equation}
$M$ is isomorphic to $(A,B)$ with $A$ an upper triangular matrix with its diagonal
$$(a,aq^{-1},\ldots,aq^{-(i-1)}, c,cq^{-1},\ldots,cq^{-(j-1)}).$$
As $a,c$ are not $q$-equivalent, $A$ is similar to
$$\diag(a,aq^{-1},\ldots,aq^{-(i-1)}, c,cq^{-1},\ldots,cq^{-(j-1)}).$$
Since $B$ satisfies that $AB=qBA$, $(A,B)$ is isomorphic to $(A_1,B_1)\oplus(A_2,B_2)$ by Lemma \ref{Jordan q-layer}. Then the sequence in (\ref{eq:LMN}) is split, which shows that
$\Ext^1_{\bbA_q^2}(N,L)=0$. Hence $\ext^1_{\bbA_q^2}(\scrD_{q,i},\scrD_{q,j})=0$.

Case (ii): \underline{$\scrV=\scrN_{q,i},\scrN=\scrD_{q,j}$}.
We can choose
$L=(A_1,B_1)\in \scrN_{q,i}$, and $N=(A_2,B_2)\in\scrN_{q,j}$ such that
$A_1=J_i(0),A_2=J_j(0)$, $B_1=\diag(b_1,qb_1,\dots,q^{i-1}b_1)$ and $B_2=\diag(b_2,qb_2,\dots,q^{j-1}b_2)$ where $b_1,b_2\in\bbC^*$ are not $q$-equivalent.
Then one can check that $\Hom_{\bbA_q^2}(L,N)=0=\Ext^1_{\bbA_q^2}(N,L)$ similar to Case (i), and so $\hom_{\bbA_q^2}(\scrN_{q,i},\scrN_{q,j})=0=\ext^1_{\bbA_q^2}(\scrN_{q,i},\scrN_{q,j})$.

Case (iii): \underline{$\scrV=\scrD_{q,i},\scrW=\scrN_{q,j}$}. We can
choose $L=(A_1,B_1)\in \scrD_{q,i}$, and $N=(A_2,B_2)\in\scrN_{q,j}$ such that
$A_1=\diag(a,aq^{-1},\ldots,aq^{-(i-1)})$ with $a\neq 0$ and $A_2=J_j(0)$.
Similar to Case (i), one can show that $\Hom_{\bbA_q^2}(L,N)=0$, and then $\hom_{\bbA_q^2}(\scrD_{q,i},\scrN_{q,j})=0$.

On the other hand, for any short exact sequence
\begin{equation}
\label{eq:LMN2}
0\rightarrow N\rightarrow M\rightarrow L\rightarrow0,
\end{equation}
$M$ is isomorphic to $(A,B)$ with $A$ an upper triangular matrix. By our assumption, $A$ is similar to
$$\diag( a,aq^{-1},\ldots,aq^{-(i-1)},J_j(0)).$$
Since $B$ satisfies that $AB=qBA$, it follows from Proposition \ref{QCmatrix} that $(A,B)$ is isomorphic to $(A_1,B_1)\oplus(A_2,B_2)$, and then the sequence in (\ref{eq:LMN2}) is split. So
$\Ext^1_{\bbA_q^2}(N,L)=0$. Hence $\ext^1_{\bbA_q^2}(\scrD_{q,i},\scrN_{q,j})=0$.

Case (iv): $\underline{\scrV=\scrN_{q,i},\scrW=\scrD_{q,j}}$. It is similar to Case (iii).

The lemma is proved.
\end{proof}

The following lemma shall be used to give the irreducible components and their dimensions for $\scrK_{q,n}$.

\begin{lem}[\cite{CB-S}]
\label{lem:C-BS}
If $\calC_i\subseteq\scrK_{q,n_i}$ are irreducible components ($1\leq i\leq t$), and $n=\sum_{i=1}^tn_i$, then $\calC=\overline{\calC_1\oplus \cdots \calC_t}$ is an irreducible component of $\scrK_{q,n}$ if and only if $\ext^1_{\bbA_q^2}(\calC_i,\calC_j)=0$ for all $i\neq j$. In this case,
\begin{equation*}
\dim \calC=\sum_{i=1}^t\dim \calC_i+\sum_{i\neq j} (n_in_j-\hom_{\bbA_q^2}(\calC_i,\calC_j)).
\end{equation*}
\end{lem}

\subsection{Definition of $\scrK_{q,(\bfm,\bfr)}$}

For any vector $\bfv=(v_1,v_2,\ldots,v_{s})\in \bbZ^{s}_{\geq0}$, we define
\begin{align*}
\|\bfv\|:=&v_1+2v_2+\ldots+sv_{s},\,\,\,|\bfv|:=v_1+v_2+\ldots+v_{s}.
\end{align*}
Introduce the indexing set
\begin{align}
\label{ML}
{\ML_{q,n}}:=
\{(\bfm,\bfr) \mid\bfm=(m_1,\ldots,m_{\ell})\in\bbN^{\ell},
\bfr=(r_1,\ldots,r_{\ell-1})\in\bbN^{\ell-1},\,\, \|\bfm\|+\|\bfr\|=n\}.
\end{align}

{Note that when $\ell=1$, $\ML_{q,n}=\{\big((n),\vec{0}\big)\}$, which has only one element.}

\begin{df}
\label{def:Kqmr}
Let $(\bfm,\bfr)\in \ML_{q,n}$. The variety $\scrK_{q,(\bfm,\bfr)}$ is defined to be
\begin{align}
\label{eqn: def of irr}
\scrK_{q,(\bfm,\bfr)}:= \ov{\scrD_{q,1}^{\oplus m_1}\oplus \cdots\oplus\scrD_{q,\ell}^{\oplus m_{\ell}}
\oplus\scrN_{q,1}^{\oplus r_1}\oplus \cdots\oplus\scrN_{q,\ell-1}^{\oplus r_{\ell-1}}}.
\end{align}
\end{df}

We have
\begin{align*}
\scrK_{q,(\bfm,\bfr)}=&\ov{(\scrD^\circ_{q,1})^{\oplus m_1}\oplus \cdots\oplus(\scrD^\circ_{q,\ell})^{\oplus m_{\ell}}
\oplus(\scrN^\circ_{q,1})^{\oplus r_1}\oplus \cdots\oplus(\scrN^\circ_{q,\ell-1})^{\oplus r_{\ell-1}}}\\
=&\ov{U_{q,1}^{\oplus m_1}\oplus \cdots\oplus U_{q,\ell}^{\oplus m_{\ell}}
\oplus V_{q,1}^{\oplus r_1}\oplus \cdots\oplus V_{q,\ell-1}^{\oplus r_{\ell-1}}}.
\end{align*}

By definition, we obtain the following remark.
\begin{rem}
\label{remark oplus}
For any $(A_1,B_1)\in \scrK_{q,(\bfm_1,\bfr_1)}$, and $(A_2,B_2)\in \scrK_{q,(\bfm_2,\bfr_2)}$, we have
$$(A_1\oplus A_2,B_1\oplus B_2)\in \scrK_{q,(\bfm_1+\bfm_2,\bfr_1+\bfr_2)}.$$
\end{rem}

Inspired by the isomorphism $\theta: \scrK_{q,n}\longrightarrow \scrK_{q^{-1},n}$, we define a bijection (also denoted by $\theta$):
\begin{align}
\theta:\ML_{q,n}\longrightarrow \ML_{q^{-1},n},
\end{align}
by mapping
\begin{align}
\notag&((m_1,\dots,m_\ell),(r_1,\dots,r_{\ell-1}))\mapsto ((r_1,\dots,r_{\ell-1}, m_\ell),(m_1,\dots,m_{\ell-1})).
\end{align}

\begin{lem}\label{lem:switch}
For any $(\bfm,\bfr) \in \ML_{q,n}$, the isomorphism $\theta:\scrK_{q,n}\rightarrow \scrK_{q^{-1},n}$ induces an isomorphism of varieties $\scrK_{q,(\bfm,\bfr)}\rightarrow
\scrK_{q^{-1},\theta(\bfm,\bfr)}$.
\end{lem}
\begin{proof}

By Definition \ref{rem: calD calN}, one can check that $\theta(\calD_{q,i})\subseteq\calN_{q^{-1},i}$, $\theta(\calN_{q,i})\subseteq \calD_{q^{-1},i}$ for any $1\leq i\leq \ell-1$, and $\theta(\calD_{q,\ell})\subseteq \calD_{q^{-1},\ell}$.
In general, for any $(\bfm,\bfr) \in \ML_{q,n}$,
\begin{eqnarray*}
&&\theta(\calD_{q,1}^{\oplus m_1}\oplus \cdots\oplus\calD_{q,\ell}^{\oplus m_{\ell}}
\oplus\calN_{q,1}^{\oplus r_1}\oplus \cdots\oplus\calN_{q,\ell-1}^{\oplus r_{\ell-1}})\\
&\subseteq& \calN_{q^{-1},1}^{\oplus m_1}\oplus \cdots\oplus\calN_{q^{-1},\ell-1}^{\oplus m_{\ell-1}} \oplus \calD_{q^{-1},\ell}^{\oplus m_\ell}
\oplus\calD_{q^{-1},1}^{\oplus r_1}\oplus \cdots\oplus\calD_{q^{-1},\ell-1}^{\oplus r_{\ell-1}}.
\end{eqnarray*}
Then
\begin{eqnarray*}
\theta(\scrK_{q,(\bfm,\bfr)})&=&\theta(\ov{\calD_{q,1}^{\oplus m_1}\oplus \cdots\oplus\calD_{q,\ell}^{\oplus m_{\ell}}
\oplus\calN_{q,1}^{\oplus r_1}\oplus \cdots\oplus\calN_{q,\ell-1}^{\oplus r_{\ell-1}}})\\
&\subseteq& \ov{\calN_{q^{-1},1}^{\oplus m_1}\oplus \cdots\oplus\calN_{q^{-1},\ell-1}^{\oplus m_{\ell-1}} \oplus \calD_{q^{-1},\ell}^{\oplus m_\ell}
\oplus\calD_{q^{-1},1}^{\oplus r_1}\oplus \cdots\oplus\calD_{q^{-1},\ell-1}^{\oplus r_{\ell-1}}}\\
&=& \scrK_{q^{-1},\theta(\bfm,\bfr)}.
\end{eqnarray*}

On the other hand, one can construct an isomorphism $\theta': \scrK_{q^{-1},n}\rightarrow \scrK_{q,n}$ by mapping $(A,B)$ to $(B,A)$. Note that $\theta'$ is the inverse of $\theta$. Similarly, we have $\theta'( \scrK_{q^{-1},\theta(\bfm,\bfr)})\subseteq \scrK_{q,(\bfm,\bfr)}$, which yields that $\theta(\scrK_{q,(\bfm,\bfr)})= \scrK_{q^{-1},\theta(\bfm,\bfr)}$. So $\theta:\scrK_{q,(\bfm,\bfr)})\stackrel{\sim}\longrightarrow\scrK_{q^{-1},\theta(\bfm,\bfr)}$.
\end{proof}


\section{$\scrK_{q,(\bfm,\bfr)}$ are irreducible components}
\label{subsec: irreducible}
In this section, we shall prove that the variety $\scrK_{q,(\bfm,\bfr)}$ introduced in Definition~\ref{def:Kqmr}
is an irreducible component of $\scrK_{q,n}$.

\subsection{The irreducibility of $\scrD_{q,i}$ and $\scrN_{q,j}$}

We construct several maps to verify $\scrD_{q,i}$ and $\scrN_{q,j}$ are irreducible varieties. 


For $1\leq i \leq \ell-1$, define a map
\begin{align}
\label{eqn:phii}
\phi_{q,i}: \GL_i\times \bbC \times \bbC^{i-1} \longrightarrow \scrD_{q,i}
\end{align}
by setting
\[
\phi_{q,i} \big( g, a, (b_1,b_2,\ldots, b_{i-1} \big) =(gAg^{-1},gBg^{-1}),
\]
where $A=\diag(a,aq^{-1},\ldots,aq^{-(i-1)})$, and
\begin{align*}
B=\begin{bmatrix}
0 & b_1    &        &              &          \\
  &  0     & b_2    &              &        \\
  &        &\ddots  &\ddots        &        \\
  &        &        &   0          & b_{i-1}  \\
  &        &        &              &  0
\end{bmatrix}.
\end{align*}

For $i=\ell$, define a map
\begin{align}
\phi_{q,\ell}: \GL_\ell(K)\times \bbC \times \bbC^{\ell} \longrightarrow \scrD_{q,\ell}
\end{align}
by setting
\[
\phi_{q,\ell} \big( g, a, (b_1,b_2,\ldots, b_{\ell} \big) =(gAg^{-1},gBg^{-1}),
\]
where $A=\diag(a,aq^{-1},\ldots,aq^{-(\ell-1)})$, and
\begin{align*}
B=\begin{bmatrix}
0   & b_1    &        &              &          \\
    &  0     & b_2    &              &        \\
    &        &\ddots   &\ddots        &        \\
    &        &        &   0          & b_{\ell-1}  \\
b_\ell &        &        &              &  0
\end{bmatrix}.
\end{align*}

For $1\leq j \leq \ell-1$,
define a map
\begin{align}
\label{eqn:psii}
\psi_{q,j}: \GL_j(K)\times \bbC^{j}  \longrightarrow \scrN_{q,j}
\end{align}
by setting
\[
\psi_{q,j} \big( g, (a_1,a_2,\ldots, a_j) \big)=(gJ_j(0)g^{-1},gBg^{-1}),
\]
where $B=\L^q(j,j,(a_1,a_2,\ldots,a_j))$.

Then $\phi_{q,i}$ and $\psi_{q,j}$ are morphisms of algebraic varieties for any $1\leq i\leq \ell$ and $1\leq j\leq \ell-1$.


\begin{prop}
\label{prop:irreducibility of <l}
The varieties $\scrD_{q,i}$ and $\scrN_{q,j}$ are irreducible for all $1\leq i\leq \ell$ and $1\leq j\leq \ell-1$.

\end{prop}

\begin{proof}

Denote by $\scrT_i \subset \GL_i\times \bbC\times \bbC^{i-1}$ the open subset consisting of the triples
\[
\big( g, a, (b_1,b_2,\dots,b_{i-1} \big)
\]
such that $a\in \bbC^*$ for any $1\leq i\leq \ell-1$. Clearly, $\scrT_i$ is irreducible.
By definition, $\scrD_{q,i}$ is the closure of $\phi_{q,i} (\scrT_i)$. Hence, $\scrD_{q,i}$ is irreducible.

For the irreducibility of $\scrD_{q,\ell}$ and $\scrN_{q,j}$ for any $1\leq j\leq \ell-1$, the proof is similar, hence omitted here.
\end{proof}

\subsection{$\scrD_{q,i}$ and $\scrN_{q,j}$ are irreducible components}

In this subsection, we prove that the varieties $\scrD_{q,i}$ and $\scrN_{q,j}$ are irreducible components.

For any $A=(a_{ij})\in M_{n\times n}(\bbC)$, denote by
\[
\chi^A=x^n+ \chi^A_{n-1} x^{n-1}+\cdots + \chi^A_1 x +\chi^A_0
\]
its characteristic polynomial.
Then $\chi_k^A$ are polynomials in the $a_{ij}$'s for any $0\leq k\leq n-1$.

For simplicity, we introduce the following notation:
\begin{align}
\label{eqn: def bix}
(b;i)_x:=\diag(b,bx,bx^2,\ldots, bx^{i-1}),\,\,\forall i\geq1.
\end{align}

\begin{lem}\label{lem:nilpotent1}
For any $(A,B)\in \scrK_{q,n}$ with
\[
A=\diag(\mathbb{J}_{1,{r_1}}(0), \mathbb{J}_{2,{r_2}}(0),\ldots,
\mathbb{J}_{{k},r_{k}}(0))\text{ such that }k<\ell,\]
then $(A,B)\in \scrK_{q,(\vec{0},\bfr)}$, where
$\bfr=(r_1,\dots,r_{k},\overbrace{0,\dots,0}^{\ell-k-1})$; see \eqref{eqn: J}.
\end{lem}

\begin{proof}

Let $A=\diag(\mathbb{J}_{1,{r_1}}(0), \mathbb{J}_{2,{r_2}}(0),\ldots,
\mathbb{J}_{{k},r_{k}}(0))$. For any $B\in \scrR_{q,A}$, it is a $k\times k$ block matrix $(B_{ij})$,
where $B_{ij}=\QLB(i,j,(B^{(1)}_{ij},B^{(2)}_{ij},\ldots,B^{(\min\{i,j\})}_{ij}))$ is a $q$-layered block matrix by Proposition \ref{QLBmatrix}.
Then there exists a permutation matrix $P$ such that
$P^{-1}BP$ is a $\frac{1}{2}k(k+1)\times \frac{1}{2}k(k+1)$ block upper-diagonal matrix
with diagonals being a rearrangement of the diagonal blocks of $B$, i.e.,
the diagonal entries in the resulting block upper-triangular matrix are
\begin{eqnarray}
\label{eqn: form of B 1}
\diag\big( B^{(1)}_{k,k};B^{(1)}_{k-1,k-1},qB^{(1)}_{k,k};\ldots;B^{(1)}_{1,1},qB^{(1)}_{2,2},\ldots,q^{k-1}B^{(1)}_{k,k} \big),
\end{eqnarray}
see \cite[Lemma 4.8]{CW18} for details. So
$$\chi^B=\prod^{k}_{i=1}\prod^{i}_{j=1}\chi^{q^{(j-1)}B^{(1)}_{ii}}.$$

Let $b_{i1},\ldots, b_{i,r_i}$ be all eigenvalues of $B_{ii}^{(1)}$ for $1\leq i\leq k$.
Let $\calU$ be the subset of $\scrR_{q,A}$ consisting of matrices $B=(B_{ij})$ with the following properties:
\begin{itemize}
\item $B\text{ is invertible}$;
\item $\text{for any eigenvalue }a (\text{ respectively }b) \text{ of }B^{(1)}_{ii} (\text{ respectively }B^{(1)}_{jj})$, $a$ and $b$ are not $q$-equivalent,
$\forall 1\leq i,j\leq k$.
\end{itemize}
Obviously, $\calU$ is a dense open set of $\scrR_{q,A}$.

For any $B\in \calU$, by our assumption, there exists an invertible matrix $Q$ such that
\begin{align*}
&Q^{-1}BQ\\
=&\diag(b_{11};\ldots, b_{1r_1}; (b_{21};2)_{q^{-1}}, \ldots, (b_{2r_2};2)_{q^{-1}}; \ldots;
(b_{k1};k)_{q^{-1}},\ldots,(b_{kr_k};k)_{q^{-1}}).
\end{align*}
Then Lemma \ref{QCmatrix} shows that
$$Q^{-1}AQ=\diag(J_1^{r_1}(0),J_2^{r_2}(0),\dots, J_k^{r_k}(0)).$$
Thus $(B,A)\in \scrK_{q^{-1},\theta(\vec{0},\bfr))}$.
By Lemma \ref{lem:switch}, we obtain that $(A,B)\in \scrK_{q,(\vec{0},\bfr)}$.
It follows that $(A,B)\in \scrK_{q,(\vec{0},\bfr)}$ for any $B\in \scrR_{q,A}=\ov{\calU}$.

\end{proof}

\subsubsection{}

\begin{lem}
\label{lem:invereigen}
 For any $(A,B)\in \scrK_{q,n}$, if $A$ is invertible, and $b$
is a nonzero eigenvalue of $B$, then $b,bq^{-1},\ldots, bq^{-(\ell-1)}$ are also eigenvalues of $B$. In particular, in this case, we have $n\geq \ell$.
\end{lem}

\begin{proof}
As $b$
is a nonzero eigenvalue of $B$, there exists a nonzero vector $\xi$ such that $B\xi=b\xi$. For any $1\leq i\leq \ell-1$,
we have $B(A^i\xi)=q^{-1}(ABA^{i-1}\xi)=\cdots=q^{-i}A^iB\xi=bq^{-i}(A^i\xi)$. Hence,
$b,bq^{-1},\ldots, bq^{-(\ell-1)}$ are eigenvalues of $B$ since $A$ is invertible.
\end{proof}


\begin{lem}
\label{lem:not include0}
For any $(A,B)\in\scrK_{q,(\bfm,\bfr)}\subseteq \scrK_{q,n}$, we have
${\rank}(A)\leq n-|\bfr|$ and ${\rank}(B)\leq n-|\bfm|+m_\ell$.
\end{lem}

\begin{proof}
For any $(A,B)\in
\calD_{q,1}^{\oplus m_1}\oplus \cdots\oplus\calD_{q,\ell}^{\oplus m_{\ell}}
\oplus\calN_{q,1}^{\oplus r_1}\oplus \cdots\oplus\calN_{q,\ell-1}^{\oplus r_{\ell-1}}$, by definition,  we have
${\rank}(A)= n-|\bfr|$ and ${\rank}(B)= n-|\bfm|+m_\ell$. Since
$$\scrK_{q,(\bfm,\bfr)}=\ov{\calD_{q,1}^{\oplus m_1}\oplus \cdots\oplus\calD_{q,\ell}^{\oplus m_{\ell}}
\oplus\calN_{q,1}^{\oplus r_1}\oplus \cdots\oplus\calN_{q,\ell-1}^{\oplus r_{\ell-1}}},$$
the desired result follows.
\end{proof}

\begin{lem}
\label{lem:not include}
We have the following:
\begin{itemize}
\item[(i)] $\scrN_{q,i} \nsubseteqq   \scrD_{q,i}$, and $\scrD_{q,i} \nsubseteqq \scrN_{q,i}$ for any $1\leq i\leq \ell-1$.
\item[(ii)] $\scrD_{q,i} \nsubseteqq \scrK_{q,(\bfm,\bfr)}\subseteq \scrK_{q,i}$ and $\scrN_{q,i} \nsubseteqq \scrK_{q,(\bfm,\bfr)}\subseteq \scrK_{q,i}$ for any $1\leq i\leq \ell-1$, and any $(\bfm,\bfr)\in\ML_{q,i}$ such that
$ \|\bfm\| +\| \bfr\| =i$ and $\| \bfm\|\cdot\| \bfr\|\neq 0$.
\item[(iii)] $\scrD_{q,\ell}\nsubseteqq \scrK_{q,(\bfm,\bfr)}$ for any $(\bfm,\bfr)\in\ML_{q,\ell}$ such that $m_\ell=0$.
\end{itemize}
\end{lem}
\begin{proof}

(i) Recall the definition of $\calD_{q,i}$ in Definition \ref{rem: calD calN}.
Then $\ov{\calD_{q,i}}=\scrD_{q,i}$. Note that for any $(A,B)\in \calD_{q,i}$, $A$ is invertible. However, for any $(A,B)\in \scrN_{q,i}$, we have $\det(A)=0$. So $\calD_{q,i}\cap \scrN_{q,i}=\emptyset$.
Hence, $\scrN_{q,i} \nsubseteqq  \scrD_{q,i}$ and $\scrD_{q,i} \nsubseteqq \scrN_{q,i}$ for any $1\leq i\leq \ell-1$.

(ii) For any $(A,B)\in \scrD_{q,j}$ ($1\leq j\leq \ell-1$), it is easy to see that $B$ is a nilpotent matrix.
Hence, $\scrN_{q,i} \nsubseteqq  \scrK_{q,(\bfm,\bfr)}$ for any $\bfm\neq \vec{0}$ by Definition \ref{rem: calD calN} (ii).
Similarly, $\scrD_{q,i} \nsubseteqq  \scrK_{q,(\bfm,\bfr)}$ for any $\bfr\neq \vec{0}$ for any $1\leq i\leq \ell-1$.

(iii) For any $(A,B)\in\calD_{q,\ell}\subseteq \scrD_{q,\ell}$, we have $\det(A)\neq0 \  \text{and}\ \det(B)\neq0$.
It follows from Lemma \ref{lem:not include0} that $\calD_{q,\ell} \cap \scrK_{q,(\bfm,\bfr)} =\emptyset$ if $m_\ell=0$.
So $\scrD_{q,\ell}\nsubseteqq \scrK_{q,(\bfm,\bfr)}$ for any $(\bfm,\bfr)\in\ML_{q,\ell}$ with $m_\ell=0$.
\end{proof}

For a partition $\nu =(\nu_1, \nu_2, \ldots)$, 
we denote by $J_\nu(a)$ the Jordan normal form of eigenvalue $a$ and of block sizes $\nu_1, \nu_2, \cdots$. 

\begin{thm}
\label{prop: irreducible components for smaller n}
$\scrD_{q,i}$ is an irreducible component of $\scrK_{q,i}$ and $\scrN_{q,j}$ is an irreducible component of $\scrK_{q,j}$ for any $1\leq i\leq \ell$, and $1\leq j\leq \ell-1$.
\end{thm}

\begin{proof}
We prove it by describing the irreducible components of $\scrK_{q,n}$ by  induction on $n$ for $n\leq \ell$.

First, if $q=+1$, i.e., $\ell=1$, it is trivial. We shall assume that $q\neq1$, i.e., $\ell\geq2$.

When $n=1$, $\scrD_{q,1}=\{(a,0) \mid a\in \bbC\}$ and
$\scrN_{q,1}=\{(0,b) \mid b\in \bbC\}$ are irreducible subvarieties of $\scrK_{q,1}$, and $\scrK_{q,1}= \scrD_{q,1}\cup \scrN_{q,1}$.
Lemma \ref{lem:not include} (i) shows that neither $\scrD_{q,1} \subseteq \scrN_{q,1}$ nor $\scrN_{q,1} \subseteq\scrD_{q,1}$, so $\scrD_{q,1}$ and $\scrN_{q,1}$ are irreducible components by Proposition \ref{prop:irreducibility of <l}.

Assume that for any $1\leq k\leq \ell$, $\scrD_{q,i}$ and $\scrN_{q,i}$ are irreducible components for all $1\leq i\leq k-1$.
By Lemma \ref{lem:C-BS} and Lemma \ref{lem:extension0}, $\scrK_{q,(\bfm,\bfr)}$ are irreducible components
for any $(\bfm,\bfr)\in \ML_{q,k}$ which satisfies $\|\bfm\|\leq k-1$ and $\|\bfr\|\leq k-1$.

{\bf Claim:}
\begin{align}
\label{eqn:<l}
\scrK_{q,k}= \scrD_{q,k}\cup \scrN_{q,k}\cup \bigcup_{\tiny\begin{array}{cc}(\bfm,\bfr)\in \ML_{q,k}\\ \|\bfm\|\leq k-1, \|\bfr\|\leq k -1\end{array}} \scrK_{q, (\bfm,\bfr)}
\end{align}
if $k<\ell$;
and
\begin{align}
\label{eqn:=l}
\scrK_{q,\ell}= \scrD_{q,\ell}\cup \bigcup_{\tiny\begin{array}{cc}(\bfm,\bfr)\in \ML_{q,k}\\ \|\bfm\|< \ell, \|\bfr\|<\ell \end{array}} \scrK_{q, (\bfm,\bfr)}
\end{align}
if $k=\ell$.

For any $(A,B)\in \scrK_{q,k}$, let $\calA=\{a_1,\ldots, a_t\}$ be the set of all unequal eigenvalues of $A$. The proof of the claim is divided into the following three cases.

{\bf Case} (i): there exist two nonempty subsets $\calA_1,\calA_2$ of $\calA$
such that $\calA_1 \sqcup\calA_2=\calA$ and there do not exist any $a_i\in \calA_1$, $a_j\in \calA_2$
with
$a_i= a_jq^{-1}$ or $a_i=a_jq$. Up to a suitable conjugation, we can assume that $A=\diag(A_1,A_2)$ with $\calA_1$ (respectively $\calA_2$) as the set of the eigenvalues for $A_1$ (respectively $A_2$). Since $AB=qBA$, $B$ is of the form $B=\diag(B_1,B_2)$ such that $A_iB_i=qB_iA_i$ for $i=1,2$. By induction, we have
$(A_1,B_1)\in\scrK_{q,(\bfm_1,\bfr_1)}$, $(A_2,B_2)\in \scrK_{q,(\bfm_2,\bfr_2)}$
for some $(\bfm_1,\bfr_1)\in \ML_{q,k_1},(\bfm_2,\bfr_2) \in \ML_{q,k_2}$ with $k_1+k_2=k$.
Then $(A,B)\in\scrK_{q,(\bfm_1+\bfm_2,\bfr_1+\bfr_2)}$.

{\bf Case} (ii): $A$ is nilpotent. By Lemma \ref{lem:nilpotent1} there exists a vector
$\bfr\in\bbN^{\ell-1}$ such that $\|\bfr\|=k$ and $(A,B)\in \scrK_{q,(\vec{0},\bfr)}$.

{\bf Case} (iii): for any $(A,B)\in \scrK_{q,k}$ and $A=\diag(J_{\lambda_1}(a), J_{\lambda_2}(aq^{-1}),\ldots, J_{\lambda_t}(aq^{-{t+1}}))$
for $a\in \bbC^*$, $\lambda_i$ is a partition of $n_i$ with $1\leq i \leq t$, $t\leq k$ and $\sum^t_{i=1}n_i=k$.

{\bf Subcase} (a): \underline{$k<\ell$}. Note that $A$ is invertible. It follows from Lemma \ref{lem:invereigen} that $B$ is nilpotent.
Then there exists a vector
$\bfr\in\bbN^{\ell-1}$ such that $\|r\|=k$ and $(B,A)\in \scrK_{q^{-1},(\vec{0},\bfr)}$.
By Lemma \ref{lem:switch}, we have $(A,B)\in \scrK_{q,(\bfm,\vec{0})}$, where $\bfm=(r_1,\dots,r_{\ell-1},0)$.

{\bf Subcase} (b): \underline{$k=\ell$}. If $t=\ell$, then $A=\diag(a, aq^{-1},\ldots, aq^{-{(\ell-1)}})$, which yields that $(A,B)\in\scrD_{q,\ell}$;
if $t<\ell$, we claim that $B$ is nilpotent for any $B\in\scrR_{q,A}$. Otherwise, since $A$ is invertible, Lemma \ref{lem:invereigen} shows that $B$ is similar to $\diag(b, bq^{-1},\ldots, bq^{-(\ell-1)})$ for some $b\in\bbC^*$. Then $A$ is nilpotent
or similar to $\diag(c, cq^{-1},\ldots, cq^{-(\ell-1)})$ for some $c$ by using Lemma \ref{lem:invereigen} again. This contradicts to that $A=\diag(J_{\lambda_1}(a), J_{\lambda_2}(aq^{-1}),\ldots, J_{\lambda_t}(aq^{-{t+1}}))$ with $t<\ell$.
Hence the claim holds. The remaining proof is same to that of Subcase (a).

Therefore, \eqref{eqn:<l}--\eqref{eqn:=l} hold.
Together with Lemma \ref{lem:not include}, we have proved that $\scrD_{q,i}$ ($1\leq i\leq \ell$), $\scrN_{q,j}$ ($1\leq j\leq \ell-1$) are irreducible components.
\end{proof}

\begin{cor}
\label{cor: irreducible}
For any $(\bfm,\bfr)\in \ML_{q,n}$, we have $\scrK_{q,(\bfm,\bfr)}$ is an irreducible component of $\scrK_{q,n}$.
\end{cor}

\begin{proof}
It follows from Lemma \ref{lem:extension0}, Lemma \ref{lem:C-BS} and Theorem \ref{prop: irreducible components for smaller n}.
\end{proof}

\section{$q$-chains}
\label{sec: further prel}

\subsection{Definition of $q$-chains}
\label{section q-relation}

The notion of $q$-chains will be used in \S\ref{sec: cover}. Recall the definition of $q$-equivalent in Definition \ref{def: l relation}.


Let $S$ be a multiset of numbers $a_1, a_2,\cdots, a_k$. We can endow $S$ a set $\calA_S$ of
diagonal matrices of the form $\diag(b_1,b_2,\cdots, b_k)$ with $b_1,b_2,\dots,b_k$ a permutation of $S$. Note that all matrices in $\calA_S$ are similar.

\begin{df}
A sequence $S$ of $a, aq^{-1}, \ldots, aq^{-(k-1)}$ with $1\leq k\leq \ell$ is called a $q$-chain of length $k$ for $a\in \bbC^*$.
\end{df}

For a $q$-chain $S$ of length $k$, we take a subset $\calA^\diamond_S \subseteq \calA_S$ consisting of all diagonal matrices of the form $\diag(a, aq^{-1}, \ldots, aq^{-(k-1)})$. In fact, if $k<\ell$, then $\calA^\diamond_S$ has only one matrix; otherwise, $\calA^\diamond_S$ has $\ell$ matrices. A matrix $A_S$ in $\calA^\diamond_S$ is called to be \emph{associated} to the $q$-chain $S$.

\begin{lem}
\label{lem:unique decomposition}
Let $a_1,a_2,\dots,a_{N}$ be nonzero numbers which are not $q$-equivalent pairwise.
Let $S=\{a_1,\ldots,a_{n_{1}};a_1q^{-1},\ldots,a_{n_{2}}q^{-1};\ldots;a_1q^{-(\ell-1)},\ldots,a_{n_{\ell}}q^{-(\ell-1)}\}$, where $a_i\leq a_N$ for any $1\leq i\leq \ell$. Then there is a unique decomposition of $S$
into disjoint union of subsets:
$$
S=\bigsqcup_{i=1}^\ell S^i
$$
where $S^i$ is consisting of $q$-chains of length $i$ for any $1\leq i\leq \ell$, such that the following property holds:
\begin{itemize}
\item for any two subsets $S^i$ and $S^{j}$, there does not exist any $q$-chain in $S^{i}\sqcup S^{j}$ with length $>\max\{i,j\}$.
\end{itemize}
\end{lem}
\begin{proof}
First, let $m_\ell$ be the number of $q$-chains with length $\ell$ in $S$. Denote by $S^\ell$ the subset of $S$ consisting of the union of the $m_\ell$ $q$-chains of length $\ell$.
Denote by $S^{\ell-1}$ the subset of $S\setminus S^\ell$ consisting of the union of the $m_{\ell-1}$ $q$-chains of length $\ell-1$.
Inductively, we can use $S\setminus (S^\ell\cup S^{\ell-1})$ to define $S^{l-2}$, and so on. From this way, for each $1\leq i\leq \ell$, we obtain a subset $S^i$ of $S \setminus(\cup_{j=i+1}^\ell S^j)$ consisting of the union of the $m_{i}$ $q$-chains of length $i$.
This decomposition satisfies our requirement.

Let $S=\bigsqcup_{i=1}^\ell {}'S^i$ be another decomposition satisfies our requirement. First, $'S^\ell\subseteq S^\ell$ by definition. If there exists a $q$-chain $b,bq^{-1},\cdots, bq^{\ell-1}$ not in $'S^\ell$. Then $bq^{-j}\notin 'S^\ell$ for any $0\leq j\leq \ell-1$ by the condition on $a_1,a_2,\dots,a_N$. It follows that $\{b,bq^{-1},\cdots, bq^{\ell-1}\}\subset \bigsqcup_{i=1}^{\ell-1} S^i$, gives a contradiction to our requirement. So $'S^\ell=S^\ell$. Consider $S\setminus S^\ell$, inductively, one can prove that $'S^i=S^i$ for $1\leq i\leq \ell-1$.
So the decomposition is unique.
\end{proof}


In fact, the above lemma shows that there is a decomposition
\begin{equation}\label{eq:unique decomposition}
S=S_{\ell,1}\sqcup \cdots \sqcup S_{\ell,m_\ell}\sqcup S_{\ell-1,1}\sqcup \cdots\sqcup S_{\ell-1,m_{\ell-1}} \sqcup\cdots \sqcup S_{1,1} \sqcup\cdots \sqcup S_{1,m_1},
\end{equation}
where $S_{i,j}$ is a $q$-chain of length $i$ for each $1\leq i\leq \ell$, $1\leq j\leq m_i$ which satisfy that
for any two subsets $S_{i,j}$ and $S_{i',j'}$, there does not exist any $q$-chain in $S_{i,j}\sqcup S_{i',j'}$ with length $>\max\{i,i'\}$.
In particular, the sets $S_{\ell,1}, \cdots , S_{\ell,m_\ell}, S_{\ell-1,1}, \cdots, S_{\ell-1,m_{\ell-1}} ,\cdots , S_{1,1} ,\cdots , S_{1,m_1}$ are unique.

The sequence of numbers $m_1,\dots,m_\ell$ appearing in \eqref{eq:unique decomposition} (which is unique) is called to be \emph{associated} to the sequence of $n_1,\dots, n_\ell$, since they only depend on $n_1,\dots,n_\ell$ by the proof of Lemma \ref{lem:unique decomposition}.

Note that when $\ell=1$, we have $m_1=n_1$.

We give an example to show how to calculate the associated sequence of numbers.
\begin{ex}\emph{
Let $\ell=4$, and $(n_1,n_2,n_3,n_4)=(3,2,3,1)$.
For a set
$$
S=\{a_1,a_2, a_3; a_1q^{-1}, a_2q^{-1}; a_1q^{-2}, a_2q^{-2}, a_3q^{-2}; a_1q^{-3}\},
$$
such that $a_1,a_2,a_3$ are not $q$-equivalent pairwise,
we have the following:
\begin{align*}
&m_4=1\ \text{and}\  S_{4,1}=\{ a_1,a_1q^{-1},a_1q^{-2},a_1q^{-3}\};\\
&m_3=1\ \text{and}\  S_{3,1}=\{a_2,a_2q^{-1},a_2q^{-2}\};\\
&m_2=0;\\
&m_1=2 \ \text{and}\   S_{1,1}=\{a_3\}, S_{1,2}=\{a_3q^{-2}\}.
\end{align*}
}
\end{ex}

\subsection{The matrices $q$-commutative with diagonal matrices}
\label{subsec: q comm with diag}
Using notations as in (\ref{eq:unique decomposition}), choose $A^\diamond_{S_{ij}}\in\calA^\diamond_{S_{ij}}$ to be a matrix associated to the $q$-chain $S_{i,j}$ for any $1\leq j\leq m_i,1\leq i\leq \ell$.
Let
$$
A^\diamond_S=\diag(A^\diamond_{S_{\ell,1}},A^\diamond_{S_{\ell,2}},\cdots, A^\diamond_{S_{\ell,m_\ell}}; A^\diamond_{S_{\ell-1,1}},A^\diamond_{S_{\ell-1,2}},\cdots,
A^\diamond_{S_{\ell-1,m_{\ell-1}}};\cdots; A^\diamond_{S_{1,1}},A^\diamond_{S_{1,2}}\cdots, A^\diamond_{S_{1,m_1}} ).
$$


By definition, for any $A'_S\in \calA_S$, there exists an invertible matrix $g$ such that $g^{-1}A'_Sg=A^\diamond_S$.

\begin{df}
\label{def:RqA}
For any $A\in M_{n\times  n}(\bbC)$, define
\begin{align}
\scrR_{q,A}:=\{B\in M_{n\times n}(\bbC)\mid AB=qBA\}.
\end{align}
\end{df}

Any $B\in \scrR_{q,A^\diamond_S}$ has the following form:
\[
\diag(B_{\ell,1},\ldots,B_{\ell,m_\ell};B_{\ell-1,1},\ldots,B_{\ell-1,m_{\ell-1}};\ldots;B_{1,1},\ldots,B_{1,m_1})
\]
such that $A^\diamond_{S_{i,j}}B_{ij}=qB_{ij}A^\diamond_{S_{i,j}}$
for any $1\leq j\leq m_i$ and $1\leq i\leq \ell$. And then
\begin{align}\label{a,aqmatrix}\tiny B_{\ell,k}=
\begin{bmatrix}
0             & b^{(1)}_{\ell,k} &               &         &                 \\
0             &     0         & b^{(2)}_{\ell,k} &         &             \\
              &               & \ddots        & \ddots  &               \\
0             &     0         &               & 0       &b^{(\ell-1)}_{\ell,k}    \\
b^{(\ell)}_{\ell,k} &     0         &  \ldots       &   0     &0
\end{bmatrix}\tiny \in M_{\ell\times\ell}(\bbC)\  \text{\small and}\  & B_{ij}= \tiny
\begin{bmatrix}
0 & b^{(1)}_{i,j} &               &         &                 \\
  &     0         & b^{(2)}_{i,j} &         &             \\
  &               & \ddots        & \ddots  &               \\
  &               &               & 0       &b^{(i-1)}_{i,j}    \\
  &               &               &         &0
\end{bmatrix}\tiny \in M_{i\times i}(\bbC),
\end{align}
for any $1\leq k\leq m_\ell$, $1\leq i\leq \ell-1$ and $1\leq j\leq m_i$.

Let $\calU_{A_S^\diamond}$ be the subset of $\scrR_{q,A^\diamond_{S}}$
consisting of $B$ with $$\rank(B_{\ell,k})=\ell,\  \text{and} \ \rank(B_{i,j})=i-1\text{ for any } 1\leq k\leq \ell, 1\leq j\leq m_i \
\text{and} \ 1\leq i\leq \ell-1.$$
From the form of $B$, we see that
$\calU_{A_S^\diamond}$ is a dense open subset.

For any $B\in \calU_{A_S^\diamond}$, we have
\begin{align}
\label{eqn: 1}
\rank(B^i)=\ell m_\ell+ \sum_{j=i+1}^{\ell-1} (j-i)m_{j},
\end{align}
for any $0\leq i\leq \ell-1$.
On the other hand,
let
$$
A_S=\diag(a_1,\ldots,a_{n_{1}}; a_1q^{-1},\ldots,a_{n_{2}}q^{-1}; \ldots; a_1q^{-(\ell-1)},\ldots,a_{n_{\ell}}q^{-(\ell-1)}).
$$
By Proposition \ref{QLBmatrix}, any $B'\in \scrR_{q, A_S}$ is of the following form:
\begin{align*}
B'=
\begin{bmatrix}
0    & B'_1   &          &        &        \\
0    &  0     & B'_2     &        &        \\
0    &        & \ddots   & \ddots &         \\
0    &  0     &  \ldots  & 0      & B'_{\ell-1} \\
B'_\ell &  0     & \ldots   & 0      & 0        \\
\end{bmatrix},
\end{align*}
where
\begin{align}
\label{eqn: Bi0}
B'_i=
\begin{bmatrix}
b_{i,1}   &          &          &             &  0 & \ldots & 0      \\
          & b_{i,2}  &          &             &  0 &\ldots  & 0    \\
          &          & \ddots   &             &  0 &\ldots  & 0       \\
          &          &          & b_{i,n_i}   &  0 & \ldots & 0 \\
\end{bmatrix}_{n_i\times n_{i+1}} \ \text{if}\  n_i\leq n_{i+1}, \
\end{align}
\begin{align}
\label{eqn: Bi1}
\text{or} \  & B'_{i}=
\begin{bmatrix}
b_{i,1}   &          &          &                   \\
          & b_{i,2}  &          &                \\
          &          & \ddots   &                    \\
          &          &          & b_{i,n_{i+1}}   \\
    0     &0         & \ldots   & 0               \\
 \vdots   &\vdots    &          &\vdots   \\
    0     &0         & \ldots   & 0              \\
\end{bmatrix}_{n_i\times n_{i+1}} \ \text{if}\  n_i\geq n_{i+1}
\end{align}
for any $1\leq i\leq \ell$. Here we set $n_{\ell+1}=n_1$.
Then $\rank(B'_i)\leq\min\{n_i,n_{i+1}\}$ for any $1\leq i\leq \ell$.

One can obtain the form of $(B')^i$ for $i\geq 2$ by a simple calculation, in particular,
\begin{align*}(B')^2=
\begin{bmatrix}
             &          & B'_1B'_2  &            &               &          \\
             &          &           & B'_2B'_3   &               &        \\
             &          &           &            & \ddots        &                      \\
             &          &           &            &               & B'_{\ell-2}B'_{\ell-1}     \\
B'_{\ell-1}B_\ell &          &           &            &               &         \\
             &B'_\ell B'_1  & \ldots    &            &                &         \\
\end{bmatrix},\cdots,
\end{align*}
\[(B')^\ell=
\diag(B'_1B'_2\cdots B'_\ell; B'_2B'_3\cdots B'_\ell B'_1; \ldots, B'_\ell B'_1B'_2\cdots B'_{\ell-1}).
\]
Then we have
$$\rank(B^i)\leq \sum_{j=1}^\ell \min\{n_j,n_{j+1},\dots,n_{j+i}\}$$
for any $B\in \scrR_{q,A_S}$,  $0\leq i\leq \ell-1$; and $\rank(B^i)\leq \ell\min\{n_1,n_2,\ldots,n_\ell\}$ for any $i\geq \ell$.

Let $\calU_{A_S}'$ be the subset of $\scrR_{q,A_{S}}$
consisting of $B$ with
\begin{align}
\label{eqn: 2}
\left\{ \begin{array}{ll}\rank(B^i)= \sum_{j=1}^\ell \min\{n_j,n_{j+1},\dots,n_{j+i}\}, \,\,\forall 1\leq i\leq \ell-1,\\
\rank(B^i)= \ell\min\{n_1,n_2,\ldots,n_\ell\},\,\,\forall i\geq \ell.\end{array}\right.
\end{align}
Then $\calU_{A_S}'$ is a dense open set of
$\scrR_{q,A_S}$.


Note that there exists a permutation matrix $P$ such that  $P^{-1}A^\diamond_{S}P=A_S$. Denote by $$P^{-1}\calU_{A^\diamond_S} P:=\{P^{-1}BP\mid B\in \calU_{A^\diamond_S}\}.$$
{Then $P^{-1}\calU_{A^\diamond_S}P$ and $\calU_{A_S}'$ are dense open set of $\scrR_{q,A_S}$. So $P^{-1}\calU_{A^\diamond_S}P\cap\calU_{A_S}'$ is a dense open set of $\scrR_{q,A_S}$. }
We have the following lemma by comparing \eqref{eqn: 1} and \eqref{eqn: 2}.



\begin{lem}\label{lem:asso number}
Given a sequence $n_1,\ldots,n_\ell$.
Let $m_1,\ldots,m_\ell$ be its associated sequence. 
Then we have
\begin{align}
\label{eqn:m}
\ell m_\ell+ \sum_{j=i+1}^{\ell-1} (j-i)m_{j}= \sum_{j=1}^\ell \min\{n_j,n_{j+1},\dots,n_{j+i}\}\,\,
\end{align}
for any $0\leq i\leq \ell-1$.
Moreover, $m_1,\ldots,m_\ell$ are uniquely determined by \eqref{eqn:m}.
\end{lem}




\section{The variety $\scrK_{q,n}$ is covered by the subvarieties $\scrK_{q,(\bfm,\bfr)}$}
\label{sec: cover}

We shall prove that $\scrK_{q,n}$ is the union of its subvarieties $\scrK_{q,(\bfm,\bfr)}$, where the pair of vectors $(\bfm,\bfr)$ run over the indexing set $\ML_{q,n}$ \eqref{ML}.

\subsection{$\scrR_{q,A}$ for diagonalizable $A$ with eigenvalues $aq^{-i}$ }

We starting with diagonalizable matrices. Recall \begin{align*}
(b;i)_x=\diag(b,bx,bx^2,\ldots, bx^{i-1}),\,\,\forall i\geq1.
\end{align*}

\begin{prop}\label{Prop:QCJor}
Let
$$
A=\diag\big(\underbrace{a,\ldots,a}_{n_1},\underbrace{aq^{-1}, \ldots, aq^{-1}}_{n_2},\ldots,\underbrace{aq^{-(\ell-1)},\ldots,aq^{-(\ell-1)}}_{n_{\ell}}\big),
$$
where $a\in\bbC^*$.
Then for any $B\in\scrR_{q,A}$, we have
$(A,B)\in \ov{\scrD_1^{\oplus m_1}\oplus \scrD_2^{\oplus m_2} \oplus \cdots\oplus\scrD_\ell^{\oplus m_\ell}}$,
where the multiset $m_1,\dots,m_\ell$ is associated to the multiset $n_1,\dots,n_\ell$.
\end{prop}

\begin{proof}
By Proposition \ref{QLBmatrix}, any $B\in \scrR_{q, A}$ is of the following form:
\begin{align*}
B=
\begin{bmatrix}
0    & B_1   &          &        &        \\
0    &  0     & B_2     &        &        \\
\vdots    &        & \ddots   & \ddots &         \\
0    &  0     &  \ldots  & 0      & B_{\ell-1} \\
B_\ell &  0     & \ldots   & 0      & 0        \\
\end{bmatrix},
\end{align*}
where
$B_i\in M_{n_i\times n_{i+1}}(\bbC)$ for any $1\leq i\leq \ell$.
%

Similar to $\calU'_{A_S}$ defined in \S\ref{subsec: q comm with diag}, let $\calV$ be a subset of $\scrR_{q,A}$ which consists of the matrices $B$ such that
$$\rank(B^i)=\sum_{j=1}^\ell \min\{n_j,n_{j+1},\dots,n_{j+i}\}$$
for any $0\leq i\leq \ell-1$, and $\rank(B^i)= \rank(B^\ell)=\ell\min\{n_1,n_{2},\dots,n_{\ell}\}$ for any $i\geq \ell$. Then $\calV$ is a dense open subset of $\scrR_{q,A}$.

For any $B\in\calV$, Lemma \ref{lem:asso number} infers that
\begin{align}
\label{eqn: vecm}
\rank(B^i)= \ell m_\ell+ \sum_{j=i+1}^{\ell-1} (j-i)m_{j}
\end{align}
for any $0\leq i\leq \ell-1$. Then
$$\rank(B^i)-\rank(B^{i+1})=m_{\ell-1}+\cdots+ m_{i+1},$$
which shows that there exist $m_{i}$ Jordan blocks of the form $J_i(0)$ in the Jordan normal form of $B$ for $1\leq i\leq \ell-1$. As $\rank(B^i)= \rank(B^\ell)=\ell\min\{n_1,n_{2},\dots,n_{\ell}\}=\ell m_\ell$ for any $i\geq \ell$, we obtain that there does not exist any Jordan block of the form $J_i(0)$ with $i\geq \ell$ in the Jordan normal form of $B$. Furthermore, $B$ has $\ell m_\ell$ nonzero eigenvalues.

Let $\bfm=(m_1,m_2,\ldots, m_{\ell})\in\bbN^{\ell}$. Then $\|\bfm\|=n$.

Let $\calV'$ be the subset of $\calV$ consisting of all matrices $B$ with $\ell m_\ell$ distinct nonzero eigenvalues. Then $\calV'$ is a dense open set of $\scrR_{q,A}$ since $\scrR_{q,A}$ is irreducible.
%
For any $B\in\calV'$, Lemma \ref{lem:invereigen} yields that the Jordan normal form of $B$ is
\begin{eqnarray*}
\diag\big((b_1;\ell)_q, (b_2;\ell)_q,\ldots,(b_{m_\ell};\ell)_q, J^{m_{\ell-1}}_{\ell-1}(0),J^{m_{\ell-2}}_{\ell-2}(0),\cdots,J^{m_1}_1(0)\big),
\end{eqnarray*}
where $b_1,\dots,b_{m_\ell}$ are not $q$-equivalent.

Hence, we have $(B,A)\in\scrK_{q^{-1},((0,\dots,0,m_\ell),(m_1,\dots,m_{\ell-1}))}$ for any $B\in\calV'$ by Lemma \ref{lem:nilpotent1} and Remark \ref{remark oplus}. It follows from Lemma \ref{lem:switch} that
 $(A,B)\in \scrK_{q,(\bfm,\vec{0})}$ for any $B\in\calV'$, and then for any $B\in \ov{\calV'}=\scrR_{q,A}$.
 \end{proof}

\subsection{$\scrR_{q,A}$ for general $A$ with eigenvalues $aq^{-i}$ }

For any vector $\vec{a}=(a_1,\dots,a_n)$,  let $J_n(\vec{a})$ denote the following matrix
\begin{align*}
J_n(\vec{a})=&
\begin{bmatrix}
 a_1     &    1                                           \\
         &a_2    &1                                        \\
         &       &   \ddots &  \ddots                      \\
         &       &          &\ddots      &  1             \\
         &       &          &            &a_n
\end{bmatrix}.
\end{align*}
In particular, if $\vec{a}=(\overbrace{a,\dots,a}^n)$,
then $J_n(\vec{a}) =J_n(a)$ is the $n\times n$ Jordan block of eigenvalue $a$.

As a generalization of Lemma~\ref{Jordan q-layer}, we obtain the following two lemmas, which will be used (only) in the proof of Proposition~\ref{prop:block a,aq^{-(l-1)}} below.

\begin{lem}\label{lem:QC1a}
For any $\vec{a}=(a_1,\dots, a_{m})$ and $\vec{b}=(a_{m+1},\dots,a_{m+n})$ with $a_i\neq qa_{m+j}$ for all $1\leq i\leq m, 1\leq j\leq n$, if
the matrix $B=(b_{ij})_{m\times n}$ satisfies that $J_m(\vec{a})B-qBJ_n(\vec{b})=0$, then $B=0$.
\end{lem}

\begin{proof}
Let $B=(B_{ij})$ be an arbitrary $m\times n$  matrix such that
\begin{align*}
&0=J_m(\vec{a})B-qBJ_n(\vec{b})=
\\
&\begin{tiny}\begin{bmatrix}
(a_1-qa_{m+1})b_{1,1}+b_{2,1}                 &  (a_1-qa_{m+2})b_{1,2}+b_{2,2}+b_{1,1}                  &    ...       & (a_1-qa_{m+n})b_{1,n}+b_{2,n}+b_{1,n-1}              \\
(a_2-qa_{m+1})b_{2,1}+b_{3,1}                  &   (a_2-qa_{m+2})b_{1,2}+b_{3,2}+b_{2,1}                 &    ...       & (a_2-qa_{m+n})b_{2,n}+b_{3,n}+b_{2,n-1}             \\
\vdots                                       &  \vdots                                                &              &   \vdots            \\
(a_{m-1}-qa_{m+1})b_{m-1,1}+b_{m,1}           &  (a_{m-1}-qa_{m+2})b_{m-1,2}+b_{m,2}+b_{m-1,1}          &    ...       & (a_{m-1}-qa_{m+n})b_{m-1,n}+b_{m,n}+b_{m-1,n-1}    \\
(a_m-qa_{m+1})b_{m,1}                         & (a_m-qa_{m+2})b_{m,2} +b_{m,1}                          &    ...       & (a_m-qa_{m+n})b_{m,n} +b_{m,n-1}
\end{bmatrix}
\end{tiny}.
\end{align*}
Since $a_i\neq qa_{m+j}$ for all $i,j$, the argument as for the anti-commuting matrices in \cite[Lemma 4.2]{CW18} also works here.
\end{proof}

The proof of the following lemma is straightforward, and we omit it here, cf. \cite[Lemma 4.3]{CW18}.

\begin{lem}\label{lem:QC2a}
For any $m,n\geq 1$, let $a_1,\dots,a_{\max (m,n)}\in \bbC^*$ with $(a_i)^\ell\neq  (a_j)^\ell$ for all $i\neq j$. Let
$\vec{a}=(a_1,\dots, a_{m})$ and $\vec{a}'=(a_1q^{-1},\dots,a_{n}q^{-1})$.
Then an $m\times n$ matrix $B=(b_{ij})$ satisfies $J_m(\vec{a})B-qBJ_n(\vec{a}')=0$ if and only if
$B$ is one of the two forms \eqref{eq:B1}--\eqref{eq:B2}:
\begin{eqnarray}
  \label{eq:B1}
B=\begin{bmatrix}
 b_{1,1}    & b_{1,2}          & \ldots                 &      b_{1,n-1}              &  b_{1,n}      \\
   0        & b_{2,2}          & \ldots                 &      b_{2,n-1}              &  b_{2,n}      \\
   0        & 0                & \ddots                 &      b_{3,n-1}              &  b_{3,n}      \\
\vdots      &                  &  \ddots                &      \ddots                 &     \vdots          \\
0           &  0               &                        &        0                    &  b_{n,n}       \\
0           &  0               & \ldots                 &        0                    &  0            \\
\vdots      & \vdots           &                        &         \vdots              & \vdots        \\
0           &  0               & \ldots                 &           0                 &  0
\end{bmatrix}
 \quad \text{if }  m\geq n,
\end{eqnarray}
where $b_{i,j}=qb_{i-1,j-1}-(a_{i-1}-a_j)b_{i-1,j}$  for $2\leq i\leq j\leq n$
and $b_{1,j}$ are arbitrary for $1\leq j\leq n$;
or
\begin{eqnarray}
  \label{eq:B2}
B=\begin{bmatrix}
b_{1,1} &  b_{1,2} &  \ldots  &  b_{1,m-1} & b_{1,m}    &  b_{1,m+1}    & \ldots        &  b_{1,n}      \\
  0     &  b_{2,2} &  \ldots  &  b_{2,m-1} & b_{2,m}    & b_{2,m+1}     & \ldots        &  b_{2,n}    \\
0       &  0       &  \ddots  &  b_{3,m-1} & b_{3,m}    & b_{3,m+1}     & \ldots        &   b_{3,n} \\
        &          &  \ddots  &  \ddots    & \vdots    &  \vdots        &               & \vdots\\
0       &  0       &          &     0      & b_{m,m}    & b_{m,m+1}     & \ldots        &   b_{m,n}
\end{bmatrix}
\quad \text{if }  m\leq n,
\end{eqnarray}
where $b_{1,j}$ are arbitrary for $n-m+1\le j \le n$ and
\begin{eqnarray*}
  \begin{cases}
  b_{i,n-m+j}=qb_{i-1,n-m+j-1}-(a_{i-1}-a_{n-m+j})b_{i-1,n-m+j}       &  \text{if} \  2\leq i\leq j\leq m, \\
  qb_{m,j}=(a_{m}-a_{j+1})b_{m,j+1}                                    &  \text{if} \  m\leq j\leq n-1,       \\
  qb_{i,j}=b_{i+1,j+1}+(a_{i}-a_{j+1})b_{i,j+1},                        &  \text{if} \   1\leq i\leq m-1, i\leq j\leq n-m-1+i.
  \end{cases}
\end{eqnarray*}
\end{lem}

\begin{rem}  \label{rem:QC1a}
Lemma \ref{lem:QC2a} reduces to Lemma~\ref{Jordan q-layer} and $B$ becomes a $q$-layered matrix if we allow $a_i=a \in \bbC^*$ for all $i$.
\end{rem}


For a partition $\mu =(\mu_1, \mu_2, \ldots)$, we denote by $\mu' =(\mu'_1, \mu'_2, \ldots)$ its \emph{transposed partition}. We also denote by $J_\mu(a)$ the Jordan normal form of eigenvalue $a$ and of block sizes $\mu_1, \mu_2, \cdots$.

Let $n=\sum^\ell_{i=1} n_i$ and $\nu_i$ be a partition of $n_i$ for each $1\leq i\leq \ell$.
Let us write the partitions in parts as
$$\nu_i=(n_{i1},n_{i2} \ldots, n_{i,s_i})\text{ and }\nu'_i=(n'_{i1},n'_{i2} \ldots, n'_{i,t_i})$$
for any $1\leq i\leq \ell$.
Let $N=\max\{t_1,t_2,\ldots, t_{\ell}\}$ and $n'_{i,t_i+1}=n'_{i,t_i+2}=\cdots=n'_{i,N}=0$.
Let $m_{j1}, m_{j2},\ldots, m_{j,\ell}$ be the associated multiset of
{$n'_{1j},n'_{2j},\ldots, n'_{\ell,j}$} for any $1\leq j\leq N$.
Set $$m_i= \sum^N_{j=1}m_{ji},\text{ for any }1\leq i\leq \ell,$$
and
$$\bfm=(m_1,m_2,\dots, m_{\ell}),\,\,\,\bfm_j=(m_{j1},m_{j2},\dots, m_{j,\ell}),\,\,\forall 1\leq j\leq N.$$

The following result generalizes Proposition~\ref{Prop:QCJor}.
\begin{prop}
\label{prop:block a,aq^{-(l-1)}}
Let $A=\diag(J_{\nu_1}(a), J_{\nu_2}(aq^{-1}),\ldots, J_{\nu_{\ell}}(aq^{-(\ell-1)}))$, where $a\in\bbC^*$.
With the notation as above,
we have $(A,B)\in \scrK_{q,(\bfm,0)}$,  for any $B\in \scrR_{q,A}$.
\end{prop}
\begin{proof}


By Lemma~\ref{Jordan q-layer}, any $B \in \scrR_{q,A}$ can be written in the same block shape as for $A$, i.e., $B=(B_{ij})_{s\times s}$, where $s=\sum_{k=1}^\ell s_k$, such that
$B_{ij}$  are block matrices of $q$-layered matrices if $\sum_{p=0}^{k-1}s_{p}\leq i< \sum_{p=0}^ks_p\leq j< \sum_{p=0}^{k+1}s_{p}$ for $1\leq k\leq \ell-1$,
or if $1\leq j < s_0+s_1$, $\sum_{p=0}^{\ell-1}s_{p} \leq i <\sum_{p=0}^\ell s_p$; while $B_{ij}=0$ for others. Here, we set $s_0=1$.

Let $a_1,\dots,a_N\in\bbC^*$ such that $(a_i)^\ell\neq (a_j)^\ell$ for any $j\neq i$.
Consider the block matrix
\[
A'=\diag(A_{1}, A_{2},\dots,A_{s_1}; A_{s_1+1}, A_{s_1+2},\dots,A_{s_1+s_2}; \dots; A_{\sum_{k=1}^{\ell-1}s_k+1},A_{\sum_{k=1}^{\ell-1}s_k+2},\dots,A_{s}),
\]
where $A_{i}=J_{n_{k,j}}(\vec{\alpha}_{n_{k,j}})$ is defined as follows:
\begin{itemize}
\item $k,j$ are the unique numbers satisfying $\sum_{p=0}^ks_p\leq i< \sum_{p=0}^{k+1}s_p,  j=i-\sum_{p=1}^ks_p$ for any $1\leq i< \sum_{p=0}^{\ell}s_p$;
\item $\vec{\alpha}_{n_{k,j}}=(a_1q^{-k},\dots, a_{n_{k,j}}q^{-k})$.
\end{itemize}
Let $B'=(B_{ij})_{s\times s}$ be of the same block shape as for $A'$.
Then $B' \in \scrR_{q,A'}$ if and only if
\[
A_iB'_{ij}-qB'_{ij}A_j=0,\quad \forall 1\leq i\leq s; 1\leq j \leq s.
\]
From Lemma~\ref{lem:QC2a} and Lemma~\ref{lem:QC1a}, we know the detailed structures of $B'_{ij}$, more explicitly, $B'_{ij}$ can be nonzero only if $\sum_{p=0}^{k-1}s_{p}\leq i< \sum_{p=0}^ks_p\leq j< \sum_{p=0}^{k+1}s_{p}$ for $1\leq k\leq \ell-1$,
or if $1\leq j <s_0+ s_1$, $\sum_{p=0}^{\ell-1}s_{p} \leq i <\sum_{p=0}^{\ell} s_p$.
Therefore, by applying Remark~ \ref{rem:QC1a} and the observation in the preceding paragraph, for any $B=(B_{ij})\in\scrR_{q,A}$, there exists $B'=(B'_{ij})\in \scrR_{q,A'}$ such that $B'$ approaches $B$ as $a_i\rightarrow a$ (e.g., by taking $a_i=a+it$ with $t\rightarrow 0$). This reduces the proof of the proposition to the following.

\vspace{2mm}
\noindent{\bf Claim.} We have $(A',B') \in\scrK_{q,(\bfm,\vec{0})}$.

Let us prove the Claim.  Denote $\diag(\overbrace{a,\ldots,a}^m)$ by $(a;m)_1$.
There exists 
 $g\in \GL_{n}$ such that
\begin{eqnarray*}
g^{-1}A'g=\diag(A''_1,\dots,A''_N),
\end{eqnarray*}
where $A''_i= \diag\big((a_i;n'_{1i})_1, (a_i q^{-1};n'_{2i})_1,\ldots, (a_i q^{-(\ell-1)};n'_{\ell i})_1\big )$ for each $1\leq i\leq N$.
Proposition \ref{QCmatrix} shows that $g^{-1}B'g \in \scrR_{q,g^{-1}A'g}$ must be of the form
\[
g^{-1}B'g =\diag(B''_{1},\dots, B''_{N} ),
\]
where $A_i''B_i''=qB_i''A_i''$ for any $1\leq i\leq N$. Together with Proposition \ref{Prop:QCJor}, we have $(A''_i,B''_i)\in \scrK_{q,(\bfm_{i},\vec{0})}$ for any $1\leq i\leq N$. By Remark \ref{remark oplus}, we conclude that $(g^{-1}A'g,g^{-1}B'g)\in\scrK_{q,(\bfm,\vec{0})}$, and then so is $(A',B')$.
The proposition is proved.
\end{proof}

\subsection{$\scrR_{q,A}$ for $A$ nilpotent}

The following proposition is a generalization of \cite[Proposition 4.6]{CW18}.

\begin{prop}\label{prop:KJsn}
For any $B\in\scrR_{q,\mathbb{J}_{s,n}(0)}$, we have $(\mathbb{J}_{s,n}(0),B)\in  \scrK_{q,(\bfm,\bfr_t)}$,
where $s=k\ell+t$ with $0\leq t<\ell$, and $\bfm=(0,\ldots,0,nk)$,
$\bfr_t=(0,\ldots,0,n,0,\ldots,0)$ denotes the vector where $n$ is its $t$-th component.
\end{prop}

\begin{proof}
The case for  $s< \ell$ is proved by Lemma \ref{lem:nilpotent1}.

We shall now assume $s\ge \ell$. Setting $B=\QLB (s,s,B_1,\ldots, B_s)$, we introduce the following variety
\begin{align*}
\calV_{s,n} &=\big\{ (\la, B_1,\ldots, B_s) \in \bbC \times M_{n\times n}(\bbC)^s \mid
\\
&\qquad\qquad
\text{ all } (sn-1)\times (sn-1) \text{ minors of } (x I_{sn} - B ) \text{ are } 0 \big\}.
\end{align*}
Let us denote $B_i=(b_{i;k,j})_{k,j}$, for each $i$. Then the characteristic polynomial of $B$ is
\[
F=\det (x I_n -B_1) \det (xI_n -qB_1)\cdots \det (x I_n -q^{(s-1)}B_1)
\]
 is a polynomial in the $x$ and $b_{1;k,j}$'s. Let $C_{sn,1}$ be the matrix obtained from $x I_{sn} - B$ with the first column and the last row deleted, and denote $G =\det C_{sn,1}$ (which is an $(sn-1)\times (sn-1)$-minor of $x I_{sn} - B$).
Observe that the product of the diagonal entries of $C_{sn,1}$ gives us (up to a power of $q$ ) a monomial $b_{2;n,1}^{s-1}\prod_{i=1}^{n-1} b_{1;i,i+1}^s,$ which is of highest degree in $b_{2;n,1}$ and is not equal to any other monomials from the expansion of $G=\det C_{sn,1}$. It follows that $(F,G)$ is a regular sequence.

Then
\[
\calV_{s,n} \subset \calV_{s,n}' := \big\{ (\la, B_1,\ldots, B_s) \in \bbC \times M_{n\times n}(\bbC)^s \mid
F=G=0 \big\}.
\]
It follows that $\dim \calV_{s,n} \le \dim \calV_{s,n}'  =(1+sn^2)-2=sn^2-1$.

Consider the composition of a projection with an isomorphism
\[
\psi: \calV_{s,n} \longrightarrow M_{n\times n}(\bbC)^s \stackrel{\cong}{\longrightarrow} \scrK_{q,\bbJ_{s,n}(0)},
\quad  (\la, B_1,\ldots, B_s) \mapsto \QLB (s,s,B_1,\ldots, B_s).
\]
Denote by $W^c$ the complement of a subvariety $W$ in $\scrR_{q,\bbJ_{s,n}(0)}$.
By definition, we have
\begin{equation*}
\text{Im}(\psi)^c =\big\{B= \QLB (s,s,B_1,\ldots, B_s) ) \mid \rank (\la I_{sn} -B) \ge sn-1, \forall \la \in \bbC \big\}.
\end{equation*}
Clearly $U_1:=\big\{B= \QLB (s,s,B_1,\ldots, B_s) ) \mid B_1 \text{ has distinct eigenvalues } b_1, \ldots, b_n \big\}$ is a dense open subset of $\scrK_{q,\bbJ_{s,n}(0)}$.
Since $\dim \ov{\text{Im}(\psi)} = \dim \text{Im}(\psi) \le \dim \calV_{s,n} \le sn^2-1 <\dim \scrR_{q,\bbJ_{s,n}(0)}$, we see
\[
\calU_{q,\mathbb{J}_{s,n}(0)} := \ov{\text{Im}(\psi)}^c \cap U_1
\]
is dense open in $\scrR_{q,\bbJ_{s,n}(0)}$. By construction we have
\begin{align}
\label{eq:Ubig}
\begin{split}
\calU_{q,\mathbb{J}_{s,n}(0)} &\subseteq \big\{B= \QLB (s,s,B_1,\ldots, B_s) ) \mid \rank (\lambda I_{sn} -B) \ge sn-1, \forall \la \in \bbC,
\\
&\qquad\qquad\qquad\qquad\qquad  B_1 \text{ has distinct nonzero eigenvalues } b_1, \ldots, b_n \big\}.
\end{split}
\end{align}

It follows by \eqref{eq:Ubig} that for any $B\in \calU_{q,\mathbb{J}_{s,n}(0)}$ there is a unique Jordan block with eigenvalue $b_i$ (respectively, $b_iq^{-1},
b_iq^{-2},\cdots, b_iq^{-(\ell-1)}$) in the Jordan normal form  of $B$, which is $J_{k+1}(b_i)$
(respectively, $J_{k+1}(b_iq^{-1}),\ldots,J_{k+1}(b_iq^{-(t-1)}),J_{k}(b_iq^{-t}),\ldots,J_{k}(b_1q^{-(\ell-1)})  $), for all $i$.
It follows that for any
$B\in \calU_{q,\mathbb{J}_{s,n}(0)}$, $B$ is similar to
\begin{align*}
\diag \big(J_{k+1}(b_1),J_{k+1}(b_1q^{-1}), \ldots, J_{k+1}(b_1q^{-(t-1)}), J_{k}(b_1q^{-t}),\ldots, J_{k}(b_1q^{-(\ell-1)});\dots;\\
             J_{k+1}(b_n),J_{k+1}(b_nq^{-1}), \ldots, J_{k+1}(b_nq^{-(t-1)}), J_{k}(b_nq^{-t}),\ldots, J_{k}(b_nq^{-(\ell-1)}) \big),
\end{align*}
for distinct $b_1, \ldots, b_n \in \bbC^*$.
Together with Lemma~\ref{lem:switch},  Proposition~\ref{prop:block a,aq^{-(l-1)}}, and that $\ov{\calU_{q,\mathbb{J}_{s,n}(0)}} =\scrR_{q,\bbJ_{s,n}(0)}$, the desired result follows.
\end{proof}


The following proposition is a generalization of Proposition \ref{prop:KJsn}.

\begin{prop}
\label{prop:Jordan block 0}
Let $A=\diag(\mathbb{J}_{1,n_1}(0),\ldots, \mathbb{J}_{s,n_s}(0))$, $n=\sum^s_{i=1} in_i$. Let
$i=k_i\ell+t_i$ with $0\leq t_i<\ell$ for any $1\leq i\leq s$, and let $p=\sum^s_{i=1}n_ik_i$. Let
$$r_j=\sum_{i=1, t_i=j}^{s} n_i,\,\,\forall 1\leq i\leq s.$$ 
and $\bfr=(r_1,r_2,\dots, r_{\ell-1})$.
Then for any $B\in \scrR_{q,A}$, we have $(A,B)\in\scrK_{q,((0,\dots,0,p),\bfr)}$.
\end{prop}

\begin{proof}
Any $B\in \scrR_{q,A}$ can be written as a block matrix $B=(B_{ij})_{s\times s}$ such that
$\mathbb{J}_{i,n_i}(0)B_{ij}-qB_{ij}\mathbb{J}_{j,n_j}(0)=0$ for any $1\leq i,j\leq s$. It follows by Lemma~\ref{lem:QLB} that each $B_{ij}$ is a $q$-layered block matrix, i.e., of the form \eqref{eq:block al1} or \eqref{eq:block al2}. Introduce the following dense open subset $\calU$ of $\scrR_{q,A}$ (and so $\scrR_{q,A} =\ov{\calU}$):
\begin{align*}
\calU = &\left\{B=(B_{ij})\in \scrR_{q,A} \mid B_{ii}\in \calU_{\mathbb{J}_{i,n_i}(0)}, B_{ii}^\ell \text{ and }B_{jj}^\ell \right.\\
&\left.\text{do not share a common eigenvalue}, \forall 1\leq i\neq j\leq s \right\},
\end{align*}
where $\calU_{q,\mathbb{J}_{i,n_i}(0)}$ is a dense open subset of $\scrR_{q,\mathbb{J}_{i,n_i}(0)}$ for each $i$ as given in Proposition~\ref{prop:KJsn}. Note the first condition in $\calU$ implies that all eigenvalues in $B_{ii}$ are distinct for each $i$, and the second condition in $\mathcal U$ implies that
 the eigenvalues of $B_{ii}$ and $B_{jj}$ for $j\neq i$ are not $q$-equivalent
Clearly $\calU$ is a dense open subset of $\scrR_{q,A}$.

Similar to \cite[Lemma 4.8]{CW18}, $B$ is similar to a block upper-triangular matrix with its diagonal $\diag(B_{11},B_{22},\dots,B_{ss})$, by doing some operations of interchanging the rows and columns. Following the argument for Proposition~\ref{prop:KJsn} (see \eqref{eq:Ubig}), we see there exists a dense open subset $\calU^{\scrD}$ of $\calU$ such that there exists a unique Jordan block  for $B\in \calU^{\scrD}$ associated to each eigenvalue.

Let $B\in \calU^{\scrD}$, and let $\{b_{ij} \mid 1\le j \le n_i\}$ be the set of eigenvalues of $B_{ii}$, for each $i$. It follows by the above discussion and by \cite[Lemma 4.8]{CW18} that there exists $P\in GL_n$ such that $P^{-1}BP=\diag \big(B'_{11},B'_{22},\ldots,B'_{ss} \big)$, where
\begin{align*}
B'_{ii}=\diag \big(J_{k_i+1}(b_{i1}),J_{k_i+1}(b_{i1}q^{-1}), \ldots, J_{k_i+1}(b_{i1}q^{-(t_i-1)}), J_{k_i}(b_{i1}q^{-t_i}),\ldots, J_{k_i}(b_{i1}q^{-(\ell-1)});\dots;\\
             J_{k_i+1}(b_{in_i}),J_{k_i+1}(b_{in_i}q^{-1}), \ldots, J_{k_i+1}(b_{in_i}q^{-(t_i-1)}); J_{k_i}(b_{in_i}q^{-t_i}),\ldots, J_{k_i}(b_{in_i}q^{-(\ell-1)}) \big),
\end{align*}
with $0\neq (b_{ij})^\ell\neq (b_{kl})^\ell$ for all $(i,j)\neq (k,l)$. It follows by Proposition \ref{QCmatrix} that any $C\in \scrR_{q,P^{-1}BP}$ is of the form
\[
\diag (C_{11},\ldots,C_{1n_1}, C_{21},\ldots,C_{2n_1}\ldots, C_{s1},\ldots,C_{sn_s})
\]
where the $i\times i$ matrix $C_{ij}$ (for $1\leq j\leq n_i, 1\leq i \leq s$) satisfies
\begin{eqnarray*}
&&C_{ij}\diag \big(J_{k_i+1}(b_{ij}),J_{k_i+1}(b_{ij}q^{-1}), \ldots, J_{k_i+1}(b_{ij}q^{-(t_i-1)}), J_{k_i}(b_{ij}q^{-t_i}),\ldots, J_{k_i}(b_{ij}q^{-(\ell-1)}) \big)=\\
&&q\diag \big(J_{k_i+1}(b_{ij}),J_{k_i+1}(b_{ij}q^{-1}), \ldots, J_{k_i+1}(b_{ij}q^{-(t_i-1)}), J_{k_i}(b_{ij}q^{-t_i}),\ldots, J_{k_i}(b_{ij}q^{-(\ell-1)}) \big) C_{ij}.
\end{eqnarray*}
By Lemma~\ref{lem:switch} and Proposition~\ref{prop:block a,aq^{-(l-1)}} we have that
\[
\big(C_{ij},\diag(J_{k_i+1}(b_{ij}),J_{k_i+1}(b_{ij}q^{-1}), \ldots, J_{k_i+1}(b_{ij}q^{-(t_i-1)}), J_{k_i}(b_{ij}q^{-t_i}),\ldots, J_{k_i}(b_{ij}q^{-(\ell-1)})) \big)
\]
is in $\scrK_{q, ((0,\dots,0,k_i),\bfr_{t_i})}$, where $\bfr_{t_i}$ denotes the vector with $n_i$ its $t_i$-th component.
By definition and Remark \ref{remark oplus}, this implies that $(C,P^{-1}BP)\in \scrK_{q,((0,\dots,0,p),\bfr)}$.
In particular, we have $(P^{-1}AP,P^{-1}BP) \in  \scrK_{q,((0,\dots,0,p),\bfr)}$ and hence $(A,B)\in \scrK_{q,((0,\dots,0,p),\bfr)}$ for any $B\in \calU^{\scrD}$. Therefore,  $(A,B)\in \scrK_{q,((0,\dots,0,p),\bfr)}$ for any $B\in \scrK_{q,A} =\ov{\calU^{\scrD}}$.

\end{proof}

\subsection{A covering of $\scrK_{q,n}$}

Now we are ready to formulate the main result of this section. Recall $\ML_{q,n}$ from \eqref{ML}.

\begin{prop}\label{prop:cover the entire space}
We have $\scrK_{q,n} =\bigcup_{(\bfm,\bfr)\in \ML_{q,n}} \scrK_{q,(\bfm,\bfr)}$. 
\end{prop}

\begin{proof}
Let $(A,B)\in \scrK_{q,n}$. The Jordan normal form of $A$ is similar to $\diag(A_0, A_1, \ldots, A_s)$, where $A_0$ is nilpotent as in Proposition~\ref{prop:Jordan block 0}, and each $A_i$ ($i\ge 1$) consists of Jordan blocks of nonzero eigenvalues in $\{a_i, a_iq^{-1},\ldots, a_iq^{-(\ell-1)}\}$ (with Jordan blocks of eigenvalues $a_i, a_iq^{-1},\ldots, a_iq^{-(\ell-1)}$ arranged as in Proposition~ \ref{prop:block a,aq^{-(l-1)}}) such that $(a_i)^\ell \neq  (a_j)^\ell$ for $i\neq j$.  Then $B=(B_0,B_1,\cdots, B_s)$ of the same block form. By Proposition~\ref{prop:Jordan block 0}, we have $(A_0,B_0) \in \scrK_{q,((0,\dots,0,p_0),\bfr_0)}$, and by Proposition~\ref{prop:block a,aq^{-(l-1)}}, we have $(A_i,B_i) \in \scrK_{q,(\bfm_i,\vec{0})}$, for $i\ge 1$ and suitable $p_0, \bfr_0$ and $\bfm_i$.
By applying Remark \ref{remark oplus} repeatedly, we obtain $(A,B)\in \scrK_{q,(\bfm,\bfr)}$, for some $\bfr$ and $\bfm$.
\end{proof}

Together with Corollary \ref{cor: irreducible}, 
we have established the following result.
\begin{thm}
\label{thm: irreducible comp}
Let $\ell<\infty$. Then the variety $\scrK_{q,n}$ has one irreducible component $\scrK_{q,(\bfm,\bfr)}$ for each $(\bfm,\bfr)\in\ML_{q,n}$ such that
$\scrK_{q,n} =\bigcup_{(\bfm,\bfr)\in \ML_{q,n}} \scrK_{q,(\bfm,\bfr)}$.
\end{thm}

For any $t\in\bbZ_{>0}$, let $p_s(t)$ be the number of partitions of $t$ in which the largest part $\leq s$; note $p_s(0)=1$ for $s\geq0$. For convenience, set $p_0(t):=0$ for any $t>0$.

\begin{cor}
\label{cor: no of irreducible}
If $q$ is a root of unity, then the number of the irreducible components in $\scrK_{q,n}$ is $\sum_{i+j= n}p_{\ell-1}(i)p_{\ell}(j)$.
\end{cor}

\begin{proof}
It follows by counting the cardinality of $\ML_{q,n}$.
\end{proof}

\section{Dimension of $\scrK_{q,(\bfm,\bfr)}$}
\label{sec: dim}
In this section, we calculate the dimension of every irreducible components $\scrK_{q,(\bfm,\bfr)}$.

\subsection{Dimension of $\scrD_{q,i}$ and $\scrN_{q,i}$}

\begin{lem}
\label{prop: dimension of generators}
Let $\ell<\infty$.
We have
$\dim \scrD_{q,i}=i^2$, $\dim \scrD_{q,\ell}=\ell^2+1$, and $\dim \scrN_{q,i}=i^2$ for any $1\leq i \leq \ell-1$.
\end{lem}

\begin{proof}
Recall the definition of $\calD_{q,i}$ in Remark \ref{rem: calD calN} (i). Then $\calD_{q,i}$ is a dense open set of $\scrD_{q,i}$.
It suffices to show that $\dim \calD_{q,i}=i^2$.

Recall the morphism $\phi_{q,i}$ in \eqref{eqn:phii}.
Denote by $\calV=\phi^{-1}_{q,i}(\calD_{q,i})$ the open set in $\GL_i\times \bbC\times \bbC^{i-1}$,
and then the restriction of $\phi_{q,i}$ to $\calV$ is also regular, which is also denoted by $\phi_{q,i}$.
Clearly, $\calD_{q,i}$ and $\calV$ are irreducible varieties, and $\dim \calV=i^2+1+i-1=i^2+i$.

For any $(A,B)\in \calD_{q,i}$, there exists $g\in\GL_i$
such that
\begin{align*}
A'&=g^{-1}Ag=\diag(a,aq^{-1},\ldots, aq^{-(i-1)});\\
B'&=g^{-1}Bg=\begin{bmatrix}
0   & b_1    &        &              &          \\
    &  0     & b_2    &              &        \\
    &        &\ddots   &\ddots        &        \\
    &        &        &   0          & b_{i-1}  \\
    &        &        &              &  0
\end{bmatrix}.
\end{align*}

A point in $\phi^{-1}_{q,i}(A',B')$ is of the form
\[
\big(h,a,(c_1,c_2,\ldots,c_{i-1})\big)\in \GL_i\times \bbC\times \bbC^{i-1}
\]
such that
\begin{eqnarray}\label{eq:ci}
h\diag(a,aq^{-1},\ldots, aq^{-(i-1)})h^{-1}=A',\ \
h \begin{bmatrix}
0   & c_1    &        &              &          \\
    &  0     & c_2    &              &        \\
    &        &\ddots   &\ddots        &        \\
    &        &        &   0          & c_{i-1}  \\
   &        &        &              &  0
\end{bmatrix}h^{-1}=B'.
\end{eqnarray}
A point in $\phi^{-1}_{q,i}(A,B)$ is of the form
\[
\big(gh,a,(c_1,c_2,\ldots,c_{i-1})\big)\in GL_i\times \bbC\times \bbC^{i-1}.
\]
So $\dim \phi^{-1}_{q,i}(A,B)= \dim \phi^{-1}_{q,i}(A',B') $.

From \eqref{eq:ci}, we have
\[
h=\diag(h_1,h_2,\ldots,h_i)
\]
with $h_j$ nonzero for any $1\leq j\leq i$, and $h_jc_j=b_jh_{j+1}$ for any $1\leq j\leq i-1$.
Then any point in $\phi^{-1}_{q,i}((A',B'))$ is of the form
\[
\big(h,a,(u_1b_1,u_2b_2,\ldots,u_{i-1}b_{i-1})\big),
\]
where $h=\diag(h_1,u_1h_1,\ldots, u_1u_2\cdots u_{i-1}h_1)$ for some nonzero numbers $u_1,\ldots, u_{i-1},h_1$.
So $\dim\phi^{-1}_{q,i}((A',B'))=i$. By a standard result in algebraic geometry, we have
$$\dim \calD_{q,i}=\dim \calV-\dim \phi^{-1}_{q,i}((A',B'))=i^2.$$

As $\scrN_{q,i}$ and $\scrD_{q^{-1},i}$ are isomorphic by Lemma \ref{lem:switch},
we have $\dim\scrN_{q,i}=i^2$ for any $1\leq i\leq \ell-1$. The proof is complete.

For the dimension of $\scrD_{q,\ell}$, its proof is similar to the above, hence omitted here.

\end{proof}

\subsection{Dimension of $\scrK_{q,(\bfm,\bfr)}$}

Using Lemma \ref{prop: dimension of generators} and Lemma \ref{lem:C-BS},  we can determine the dimension of  $\scrK_{q,(\bfm,\bfr)}$.

\begin{thm}
\label{prop: dimension of Cl}
For any $(\bfm,\bfr)\in\ML_{q,n}$, we have $\dim \scrK_{q,(\bfm,\bfr)}=n^2+m_\ell$.
\end{thm}

\begin{proof}
By definition,
the variety $\scrK_{q,(\bfm,\bfr)}$ is defined to be
$$
\ov{\scrD_{q,1}^{\oplus m_1}\oplus \cdots\oplus\scrD_{q,\ell}^{\oplus m_\ell}\oplus
\scrN_{q,1}^{\oplus r_1}\oplus \cdots\oplus\scrN_{q,\ell-1}^{\oplus r_{\ell-1}}}.
$$
From Proposition \ref{prop:cover the entire space} and Lemma \ref{lem:C-BS}, we have
\begin{align*}
&\dim \scrK_{q,(\bfm,\bfr)} \\
=&\sum_{i=1}^{\ell }m_i\dim \scrD_{q,i}+ \sum_{i=1}^{\ell -1} r_i\dim\scrN_{q,i} + m_\ell (m_\ell -1)\ell ^2+\sum_{i=1}^{\ell -1}(\ell m_\ell) (im_i)
+\sum_{i=1}^{\ell -1}(\ell m_\ell) (ir_i) \\
&+ \sum_{i=1}^{\ell -1} (im_i)(\ell m_\ell) +\sum_{j=1,j\neq i}^{\ell -1} \sum_{i=1}^{\ell -1} (im_i)(jm_j)+ \sum_{i=1}^{\ell -1} m_i(m_i-1)i^2+ \sum_{j=1}^{\ell -1} \sum_{i=1}^{\ell -1} (im_i)(jr_j)\\
&+ \sum_{i=1}^{\ell -1} (ir_i)(\ell m_\ell) +\sum_{j=1}^{\ell -1} \sum_{i=1}^{\ell -1} (ir_i)(jm_j)+ \sum_{j=1,j\neq i}^{\ell -1} \sum_{i=1}^{\ell -1} (ir_i)(jr_j)+\sum_{i=1}^{\ell -1} r_i(r_i-1)i^2
\end{align*}
since for any $\scrV,\scrW\in \{ \scrD_{q,\ell },\scrD_{q,1}\dots,\scrD_{q,\ell -1}, \scrN_{q,1} ,\dots, \scrN_{q,\ell -1} \}$ we have
$\hom_{\bbA_q^2}(\scrV,\scrW)=0$ by Lemma \ref{lem:extension0}.
It follows from Proposition \ref{prop: dimension of generators} that
\begin{eqnarray*}
\dim \scrK_{q,(\bfm,\bfr)} &=& (m_\ell\ell +\sum_{i=1}^{\ell -1} (im_i)+\sum_{i=1}^{\ell -1} r_ii)^2+m_\ell\\
&=&n^2+m_\ell,
\end{eqnarray*}
by noting that $\ell m_\ell +\sum_{i=1}^{\ell -1} im_i+\sum_{i=1}^{\ell -1} i r_i=n$.


\end{proof}










\section{GIT quotients}

 \label{sec:GIT}

In this section we shall study the GIT quotient $\scrK_{q,n}//GL_n$.
%
%
\subsection{Preliminaries on GIT quotients}

The group $G=GL_n$ acts on
\[
X_n=M_n(\bbC)\times M_n(\bbC)
\]
by letting
\[
g\cdot (A,B) =(g Ag^{-1},gBg^{-1}), \qquad
\text{ for }g\in G,(A,B)\in X_n.
\]
Denote by $\bbC[X_n]$ the coordinate ring of $X_n$, and by $\bbC[X_n]^G\subseteq \bbC[X_n]$ the subring of invariants. By a theorem of Hilbert, $\bbC[X_n]^G$ is a finitely generated $\bbC$-algebra since $G$ is reductive. The following result is due to Gurevich, Procesi and Sibirskii; see \cite[Theorem 2.7.9]{Sch07}.

\begin{prop}
  \label{prop:trace}
The invariant ring $\bbC[X_n]^G$ is generated by $\mathrm{Trace}(A^iB^j)$,  for all $i,j\geq0.$
\end{prop}
We consider the GIT quotient of $X_n$ with respect to the action of $G$. A standard reference of geometric invariant theory (GIT) is the book \cite{MFK94}. 
Define
\[
X_n// G:=\mathrm{Specmax}(\bbC[X_n]^G).
\]
The inclusion
$\bbC[X_n]^G\subseteq \bbC[X_n]$ gives rise to a $G$-invariant morphism
\[
\pi: X_n\longrightarrow X_n// G.
\]
Some fundamental results of GIT in our setting are summarized below (see \cite{MFK94}).
\begin{thm}
 \label{fundamental theorem of GIT}
 The following statements hold.
\begin{itemize}
\item[(1)]   If $W_1$ and $W_2$ are two disjoint non-empty $G$-invariant closed subsets of $X_n$, then there is a $G$-invariant function $f\in \bbC[X_n]^G$ such that $f|_{W_1}\equiv1$ and $f|_{W_2}\equiv0$. In particular, the images of $W_1$ and $W_2$ under $\pi$ are disjoint.
\item[(2)] Let $v\in X_n$. Then the orbit closure $\ov{G\cdot v}$ contains a unique closed orbit.
\item[(3)] The map $\pi: X_n\rightarrow X_n// G$ induces a bijection between the set of closed orbits in $X_n$ and the points of $X_n// G$.
\end{itemize}
\end{thm}


%
%
\subsection{The GIT quotient $\scrK_{q,n}//\GL_n$}

Since $\scrK_{q,n}\subseteq X_n$ is a $G$-invariant closed subset of a $G$-module, the surjection $\bbC[X_n]\rightarrow \bbC[\scrK_{q,n}]$ is $G$-equivariant and this leads to the following commutative diagrams
\[
\xymatrix{ \bbC[X_n]^{G} \ar@{->>}[r] \ar@{^(->}[d] & \bbC[\scrK_{q,n}]^{G} \ar@{^(->}[d] \\
\bbC[X_n] \ar@{->>}[r]  & \bbC[\scrK_{q,n}] }
 \qquad\qquad
 \xymatrix{ \scrK_{q,n} \ar@{^(->}[r] \ar@{->>}[d] & X_n \ar@{->>}[d] \\
\scrK_{q,n}//{G} \ar@{^(->}[r]  & X_n//{G}
 }\]


Introduce the indexing set
\begin{equation}
\label{TPL}
\TPL_{q,n} :=\{(p,m,r) \in \bbZ_{\geq0}^3\mid \ell p+m+r=n \}.
\end{equation}
Note that if $\ell=1$, i.e., $q=1$, we have $\TPL_{q,n}=\{(n,0,0) \}$.

\subsection{}
In this subsection, we assume that $1\neq\ell<\infty$. For any $(p,m,r)\in\TPL_{q,n}$, define
\begin{align}
\label{def:Zpmr}
\scrZ_{p,m,r}&:= \ov{\scrD_{q,\ell}^{\oplus p}\oplus\scrD_{q,1}^{\oplus m}\oplus \scrN_{q,1}^{\oplus r}},\\
\calU_{p,m,r}&:=\{(A,B)\in U_{q,\ell}^{\oplus p}\oplus U_{q,1}^{\oplus m} \oplus V_{q,1}^{\oplus r}\mid {\rm rank }A=\ell p+m, {\rm rank}B=\ell p+r,\\ &\notag \qquad \text{ and the nonzero eigenvalues of }A,B\text{ are not }q\text{-equivalent}. \}\label{def:Upmr}
\end{align}
Then $\calU_{p,m,r}$ is a nonempty open subset of $\scrZ_{p,m,r}$ by Definition \ref{rem: calD calN}.

Obviously, $\scrZ_{p,m,r}=\scrK_{(m,0,\dots,0,p),(r,0,\dots,0)}$ which is an irreducible component of $\scrK_{q,n}$.

\begin{lem}
  \label{lem:U polystable}
  For any $(p,m,r)\in\TPL_{q,n}$, and $(A,B)\in  \calU_{p,m,r}$, the ${G}$-orbit of $(A,B)$ in $\scrK_{q,n}$ is closed.
\end{lem}

\begin{proof}
Let $(A,B)\in  \calU_{p,m,r}$. By definition, there exists some $g\in {G}$ such that
\begin{eqnarray}
gAg^{-1} &=&{\rm diag}((a_1,\ell)_q, \dots, (a_{p},\ell)_q,a_{p+1},\dots,a_{p+m},\underbrace{0,\dots,0}_r), \label{eq:AST1}\\
gBg^{-1} &=& {\rm diag}(C_1,\dots,C_p,\underbrace{0,\dots,0}_{m}, b_{p+1},\dots ,b_{p+r}),\label{eq:AST2}
\end{eqnarray}
where $C_i$ is invertible and of the form as \eqref{eqn: form B} shows
for $1\leq i\leq p$, and $b_{p+1}, \ldots, b_{p+r}$ are distinct nonzero scalars.

Consider the quantum plane $\bbA_q^2$. Clearly, the $\bbA_q^2$-module corresponding to $(A,B)$, or equivalently to $(gAg^{-1}, gBg^{-1})$, is semisimple by the description \eqref{eq:AST1}--\eqref{eq:AST2}, and hence its orbit is closed; see \cite[\S 12.6]{Ar69}.
\end{proof}

\begin{lem}
\label{lem: closed 1}
For any $(A,B)\in\scrK_{q,n}$, if the ${G}$-orbit of $(A,B)$ is closed, then there exists $(p,m,r)\in\TPL_{q,n}$ such that $(A,B)\in\scrZ_{p,m,r}$.
\end{lem}

\begin{proof}
The proof is divided into the following three cases.

{ Case (i):} $A$ is nonzero nilpotent.
Assume that
$$A=\diag(\mathbb{J}_{1,{r_1}}(0), \mathbb{J}_{2,{r_2}}(0),\ldots,
\mathbb{J}_{{k},r_{k}}(0))$$
so that $r_k\neq0$.

For any $B\in \scrR_{q,A}$, it is a $k\times k$ block matrix $(B_{ij})$,
where
$$B_{ij}=\QLB(i,j,(B^{(1)}_{ij},B^{(2)}_{ij},\ldots,B^{(\min\{i,j\})}_{ij}))$$
is a $q$-layered block matrix by Proposition \ref{QLBmatrix}.
Similar to the proof of Proposition \ref{prop:Jordan block 0}, there exists a permutation matrix $P$ such that
$P^{-1}BP$ is a $\frac{1}{2}k(k+1)\times \frac{1}{2}k(k+1)$ block upper-diagonal matrix
with diagonals being a rearrangement of the diagonal blocks of $B$, i.e.,
the diagonal entries in the resulting block upper-triangular matrix are
\begin{eqnarray}
\diag\big( B^{(1)}_{k,k};B^{(1)}_{k-1,k-1},qB^{(1)}_{k,k};\ldots;B^{(1)}_{1,1},qB^{(1)}_{2,2},\ldots,q^{k-1}B^{(1)}_{k,k} \big),
\end{eqnarray}
and $P^{-1}AP$ is a strict upper matrix with diagonal entries zero. In fact, if considering $P^{-1}AP$ in the same block form of $P^{-1}BP$, then its first row is $(0;0,I_{r_k};\dots; \overbrace{0,\dots, 0}^{k})$, and its first column is zero. Let $g=\diag(tI_k,I_{n-k})$ with $t$ nonzero. Considering $g (P^{-1}AP)g^{-1}$ and  $g (P^{-1}BP)g^{-1}$, it is easy to see that the ${G}$-orbit of $(A,B)$ is not closed.

{ Case (ii):} $A$ is invertible. Then we can assume that
\begin{align}
A=\diag\{ A_1,\dots,A_k\},
\end{align}
such that $\lambda_i$ and $\lambda_j$ are not $q$-equivalent for any eigenvalues $\lambda_i$ of $A_i$ and $\lambda_j$ of $A_j$ with $i\neq j$.
Then $B=\diag\{B_1,\dots,B_k\}$ such that $A_iB_i=qB_iA_i$ for any $1\leq i\leq k$.
Since the orbit of $(A,B)$ is closed, we obtain that the orbit of $(A_i,B_i)$ is closed for each $1\leq i\leq k$.

Without losing of generality, we assume that the set of the eigenvalues of $A$ contained in $\{a,aq,\dots, aq^{\ell-1}\}$. As in Proposition \ref{prop:block a,aq^{-(l-1)}},
let $A=\diag(J_{\nu_1}(a), J_{\nu_2}(aq^{-1}),\ldots, J_{\nu_{\ell}}(aq^{-(\ell-1)}))$, where $a\in\bbC^*$.
Here $\nu_i=(\nu_{i1},\nu_{i2},\cdots)$ is a partition for any $1\leq i\leq \ell$. Then
\begin{align*}
B=
\begin{bmatrix}
0    & B_1   &          &        &        \\
0    &  0     & B_2     &        &        \\
\vdots    &        & \ddots   & \ddots &         \\
0    &  0     &  \ldots  & 0      & B_{\ell-1} \\
B_\ell &  0     & \ldots   & 0      & 0        \\
\end{bmatrix},
\end{align*}
with $B_i$ a $|\nu_i|\times |\nu_{i+1}|$ matrix for $1\leq i\leq \ell$. Here $\nu_{\ell+1}=\nu_1$. It follows that
$$\rank (B)=\sum_{i=1}^\ell\rank (B_i),$$
and $B$ is invertible only if $|\nu_i|=|\nu_j|$ for any $1\leq i,j\leq \ell$.

The proof is divided into the following cases.

{ Case (ii-1):} $B$ is zero. Then $(A,B)\in \scrK_{0,n,0}$.

{ Case (ii-2):} $B$ is a nonzero nilpotent. By  Case (i) and \eqref{def: theta}, the orbit of $(A,B)$ is not closed.

{ Case (ii-3):} $B$ is invertible. Then $|\nu_i|=|\nu_j|$ for any $1\leq i,j\leq \ell$ by the above. It yields that $(A,B)\in\scrZ_{n/|\nu_1|, 0,0}$ by Proposition \ref{prop:block a,aq^{-(l-1)}}.

{ Case (ii-4):} $B$ is neither invertible nor nilpotent. Then there exists an invertible matrix $U$ such that
$UBU^{-1}=\begin{bmatrix} D_1& 0\\0&D_2 \end{bmatrix}$ with $D_1$ invertible and $D_2$ nilpotent. Then $UAU^{-1}=\begin{bmatrix} C_1& 0\\0&C_2 \end{bmatrix}$ by Lemma \ref{Jordan q-layer}. Note that $C_1,C_2$ are invertible by the assumption. So $(A,B)\cong (C_1,D_1)\oplus(C_2,D_2)$ as $\bbA_q^2$-modules. Since the orbit of $(A,B)$ is closed, the orbit of $(C_i,D_i)$ is closed for $i=1,2$.
Then $D_2=0$ by Case (i). It follows that $(C_1,D_1)\in \scrZ_{p_1,m_1,r_1}$ by Case (ii-3),  and $(C_2,D_2)\in\scrZ_{p_2,m_2,r_2}$ by Case (ii-1) for some $(p_1,m_1,r_1)$ and $(p_2,m_2,r_2)$.
Then $(A,B)\in \scrK_{p_1+p_2,m_1+m_2,r_1+r_2}$.

{ Case (iii):} $A$ is neither invertible nor nilpotent. The proof is similar to Case (ii-4), and hence omitted. The lemma is proved.
\end{proof}

For $(p,m,r) \in \TPL_{q,n}$, we denote by
\[
\pi_{p,m,r}: \scrZ_{p,m,r}\longrightarrow \scrZ_{p,m,r}//{G}
\]
the natural map induced by $\bbC[\scrZ_{p,m,r}]^{G}\subseteq\bbC[\scrZ_{p,m,r}].$
Theorem \ref{fundamental theorem of GIT}, Lemma \ref{lem: closed 1} and Theorem \ref{thm: irreducible comp} imply that $\scrZ_{p,m,r}//{G}$ is an irreducible subvariety of $\scrK_{q,n}//{G}$ for each $(p,m,r) \in \TPL_{q,n}$, and
\begin{align}
\label{eqn: cover}
\scrK_{q,n}//{G}= \bigcup_{(p,m,r)\in\TPL_{q,n}} \scrZ_{p,m,r}//{G}.
\end{align}

Define  $\widetilde{\calU}_{p,m,r}:=\pi_{p,m,r}(\calU_{p,m,r})$. Then  $\widetilde{\calU}_{p,m,r}$ is a nonempty open subset in $\scrZ_{p,m,r}//{G}$.

\begin{prop}
 \label{prop:pure dimension}
If $\ell<\infty$, then 
every irreducible component  $\scrZ_{p,m,r}//{G}$ has dimension $n+(2-\ell)p$, for any $(p,m,r) \in \TPL_{q,n}$.
\end{prop}

\begin{proof}
First, we assume $\ell\neq1$, i.e., $q\neq+1$.
By \eqref{eqn: cover}, we only need to prove that $\dim_\bbC \scrZ_{p,m,r}//{G}=n$ for any $(p,m,r) \in \TPL_{q,n}$.

Note that
\[
\pi_{p,m,r}:\calU_{p,m,r}\longrightarrow \widetilde{\calU}_{p,m,r}
\]
is a surjective regular map between irreducible varieties. Since the ${G}$-orbit of any point $(A,B)\in \calU_{p,m,r}$ is closed by Lemma~\ref{lem:U polystable}, 
Theorem \ref{fundamental theorem of GIT}(1) shows that the fiber of $\pi_{p,m,r}$ over $\pi_{p,m,r} (A,B) \in \widetilde{\calU}_{p,m,r}$ is isomorphic to the closed ${G}$-orbit of $(A,B)$ in $\calU_{p,m,r}$.

Let $(A,B)\in \calU_{p,m,r}$.
Similar to the proof of \cite[Proposition~3.4]{CW18}, the stabilizer subgroup ${G}_{(A,B)}$ has dimension $n-(\ell-1)p$.
So
\[
\dim {G}\cdot (A,B)=\dim {G}-\dim {G}_{(A,B)}=n^2-n+(\ell-1)p.
\]
Hence, by Theorem \ref{prop: dimension of Cl}, we have
\[
\dim \widetilde{\calU}_{p,m,r}=\dim \calU_{p,m,r}-\dim {G}\cdot (A,B)=(n^2+p)-(n^2-n+(\ell-1)p)=n+(2-\ell)p.
\]
Since $\widetilde{\calU}_{p,m,r}$ is dense open in $\scrZ_{p,m,r}//{G}$, we have $\dim (\scrZ_{p,m,r}//{G})=n+(2-\ell)p$.

If $\ell=1$, i.e., $q=1$, we have $\scrK_{q,n}$ is irreducible, and then so is $\scrK_{q,n}//G$. Similar to the above, it is of dimension $2n=n+(2-1)p$ by noting $p=n$ in this case.
\end{proof}

Note that if $q=-1$, then $\scrK_{q,n}//G$ is of pure dimension $n$, see \cite{CW18}.



\begin{prop}
\label{prop: irr 2}
For each $(p,m,r)\in\TPL_{q,n}$, $\scrZ_{p,m,r}//{G}$ is an irreducible component of $\scrK_{q,n}//{G}$ and  $\scrK_{q,n}//{G}= \bigcup_{(p,m,r)\in\TPL_{q,n}} \scrZ_{p,m,r}//{G}$.
\end{prop}

\begin{proof}
Recall $\calU_{p,m,r}$ defined in \eqref{def:Upmr}. It follows from Lemma \ref{lem:U polystable} that the orbit of any point in $\calU_{p,m,r}$ is closed. Claim:  $\calU_{p,m,r}\nsubseteqq \calU_{p',m',r'}$ for any $(p,m,r)\neq(p',m',r')\in\TPL_{q,n}$. Otherwise, we have $\scrZ_{p,m,r}\subseteq \scrZ_{p',m',r'}$ by taking Zariski closures for both sides, gives a contradiction to Theorem \ref{thm: irreducible comp}.
So we have $\scrZ_{p,m,r}//G\nsubseteqq\scrZ_{p',m',r'}//G$ for any $(p,m,r)\neq(p',m',r')\in\TPL_{q,n}$. The desired result follows from \eqref{eqn: cover}.
\end{proof}

\section{$q$ is not a root of unity}
\label{section:infinite}

In this section, we consider the case $\ell=\infty$, i.e., $q$ is not a root of unity. The results on the irreducible components and their dimensions of $\scrK_{q,n}$ and $\scrK_{q,n}//G$ in this case is completely same to the corresponding  result on the case when $n<\ell<\infty$. We list these results for convenience in the following.

\subsection{}

As in \S\ref{subsec: def DN}, one can define $\scrD_{q,i}$ and $\scrN_{q,i}$ for any $i\geq1$.
For any vector $\bfv=(v_1,v_2,\ldots)\in \bbN^{\infty}$ (with only finitely many components not zero), we define
\begin{align*}
\|\bfv\|:=&v_1+2v_2+\cdots,\,\,\,|\bfv|:=v_1+v_2+\cdots.
\end{align*}
Similar to \eqref{ML} and Definition \ref{def:Kqmr}, we introduce the indexing set
\begin{align*}
{\ML_{q,n}}:=
\{(\bfm,\bfr) \mid\bfm=(m_1,m_2,\ldots)\in\bbN^{\infty},
\bfr=(r_1,r_2\ldots)\in\bbN^{\infty},\,\, \|\bfm\|+\|\bfr\|=n\},
\end{align*}
and define
\begin{align*}
\scrK_{q,(\bfm,\bfr)}:= \ov{\scrD_{q,1}^{\oplus m_1}\oplus \scrD_{q,2}^{\oplus m_{2}}\oplus\cdots
\oplus\scrN_{q,1}^{\oplus r_1}\oplus\scrN_{q,2}^{\oplus r_{2}}\oplus \cdots}
\end{align*}
for any $(\bfm,\bfr)\in \ML_{q,n}$.

Using the same proofs of Lemma \ref{lem:not include} and Proposition \ref{prop: irreducible components for smaller n}, we obtain the following two lemmas.
\begin{lem}
If $\ell=\infty$, we have the following:
\begin{itemize}
\item[(i)] $\scrN_{q,i} \nsubseteqq   \scrD_{q,i}$, and $\scrD_{q,i} \nsubseteqq \scrN_{q,i}$ for any $i\geq1$.
\item[(ii)] $\scrD_{q,i} \nsubseteqq \scrK_{q,(\bfm,\bfr)}\subseteq \scrK_{q,i}$ and $\scrN_{q,i} \nsubseteqq \scrK_{q,(\bfm,\bfr)}\subseteq \scrK_{q,i}$ for any $i\geq1$, and any $(\bfm,\bfr)\in\ML_{q,i}$ such that
$ \|\bfm\| +\| \bfr\| =i$ and $\| \bfm\|\cdot\| \bfr\|\neq 0$.
\end{itemize}
\end{lem}

\begin{lem}
If $\ell=\infty$, then
$\scrD_{q,i}$, $\scrN_{q,i}$ are irreducible components of $\scrK_{q,i}$ for any $i\geq1$.
\end{lem}

In fact, using the same proof of \eqref{eqn:<l} in Proposition \ref{prop: irreducible components for smaller n}, we can obtain
\begin{align*}
\scrK_{q,k}= \scrD_{q,k}\cup \scrN_{q,k}\cup \bigcup_{\tiny\begin{array}{cc}(\bfm,\bfr)\in \ML_{q,k}\\ \|\bfm\|\leq k-1, \|\bfr\|\leq k -1\end{array}} \scrK_{q, (\bfm,\bfr)}
\end{align*}
for any $k\geq1$. So we have the following proposition.

\begin{prop}
\label{prop: irreducible l infty}
If $q$ is not a root of unity, then the variety $\scrK_{q,n}$ has one irreducible component $\scrK_{q,(\bfm,\bfr)}$ for each $(\bfm,\bfr)\in\ML_{q,n}$ such that
$\scrK_{q,n} =\bigcup_{(\bfm,\bfr)\in \ML_{q,n}} \scrK_{q,(\bfm,\bfr)}$.
\end{prop}

In order to count the cardinality of $\ML_{q,n}$, we need introduce the following notation. Let $p(n)$ the number of the partitions of $n$. In particular, for convenience, $p(0):=1$.

\begin{cor}
\label{cor: no of l infty}
If $q$ is not a root of unity, then the number of irreducible components of $\scrK_{q,n}$ is $\sum_{i+j=n}p(i)p(j)$.
\end{cor}
\begin{proof}
It follows by counting the cardinality of $\ML_{q,n}$.
\end{proof}

Using the same proofs of Lemma \ref{prop: dimension of generators} and Theorem \ref{prop: dimension of Cl}, we can prove the following result.
\begin{lem}
Let $\ell=\infty$.
We have
\begin{itemize}
\item[(i)] $\dim \scrD_{q,i}=i^2=\scrN_{q,i}$ for each $i\geq1$;
\item[(ii)] $\scrK_{q,n}$ is of pure dimension $n^2$.
\end{itemize}
\end{lem}

\begin{cor}
\label{cor: q=1}
The following are equivalent:
\begin{itemize}
\item[(a)] $\scrK_{q,n}$ is irreducible for any $n\geq1$;
\item[(b)] $\scrK_{q,n}$ is irreducible for some $n\geq1$;
\item[(c)] $q=1$.
\end{itemize}
\end{cor}

\begin{proof}
It follows from Corollary \ref{cor: no of l infty} and  Corollary \ref{cor: no of irreducible}.
\end{proof}

\subsection{}
In this subsection, we consider the GIT quotient $\scrK_{q,n}//G$. Inspired by \eqref{TPL}, we define
\begin{equation*}
\TPL_{q,n} :=\{(0,m,r) \in \bbZ_{\geq0}^3\mid m+r=n \}.
\end{equation*}

For any $(0,m,r)\in\TPL_{q,n}$, define
\begin{align*}
\scrZ_{0,m,r}&:= \ov{\scrD_{q,1}^{\oplus m}\oplus \scrN_{q,1}^{\oplus r}}.
\end{align*}
Similar to the case $\ell<\infty$, we have the following lemma.

\begin{lem}
\label{lem: l=0 irr}
For each $(0,m,r)\in\TPL_{q,n}$, $\scrZ_{0,m,r}//{G}$ is an irreducible component of $\scrK_{q,n}//{G}$ and  $\scrK_{q,n}//{G}= \bigcup_{(0,m,r)\in\TPL_{q,n}} \scrZ_{0,m,r}//{G}$.
\end{lem}
\begin{lem}
 \label{lem:pure dimension}
The variety $\scrK_{q,n}//{G}$ is of pure dimension $n$; that is, every irreducible component  $\scrZ_{0,m,r}//{G}$ has the same dimension $n$, for any $(0,m,r) \in \TPL_{q,n}$.
\end{lem}

Combining with  Proposition \ref{prop: irr 2}, Proposition \ref{prop:pure dimension}, Lemma \ref{lem: l=0 irr} and Lemma \ref{lem:pure dimension}, we have obtained Theorem \ref{main thm 2}.


\begin{thebibliography}{AAA}
\bibitem[Ar69]{Ar69}
M. Artin,
\emph{On Azumaya algebras and finite dimensional representations of rings},
J. Algebra {\bf 11} (1969), 532--563.

\bibitem[Boz16]{Boz16}
T. Bozec,
{\em     Quivers with loops and generalized crystals},
Compositio Math. {\bf 152} (2016), 1999--2040.

\bibitem[BG02]{BG02} K.A. Brown and K.R. Goodearl, Lectures on Algebraic Quantum Groups, Birkh\"{a}user,
2002.
\bibitem[CS02]{CB-S} W. Crawley-Boevey and J. Schr\"{o}er, \emph{Irreducible components of varieties of modules}, J. Reine Angew. Math. {\bf553} (2002), 201--220.

\bibitem[CW18]{CW18} X. Chen and W. Wang,
{\em Anti-Commuting varieties}, to appear in Trans. Amer. Math. Soc.,
\href{https://arxiv.org/abs/1805.00378}{arXiv:1805.00378}v2.
\bibitem[Ger61]{Ger61} M. Gerstenhaber,
{\em On dominance and varieties of commutating matrices}, Ann. of Math. {\bf 73} (1961), 324--348.
\bibitem[GH18]{GH18} K.R. Goodearl and B. Huisgen-Zimmermann, \emph{Closure in varieties of representations and irreducible components}, Algebra \& Number Theory {\bf 12} (2018), no. 2, 379--410.

\bibitem[KN87]{KN87} A.A. Kirillov and Y.A. Neretin, The variety $A_n$ of $n$-dimensional Lie algebra structures, Transl. Am. Math. Soc. {\bf137} (1987) 21--30.


\bibitem[Lu91]{Lu91} G. Lusztig,
{\em Quivers, perverse sheaves, and quantized enveloping algebras},
 J. Amer. Math. Soc. {\bf 4} (1991), 365--421.

\bibitem[MT55]{MT55} T.S. Motzkin, O. Taussky,
{\em Pairs of matrices with property L. II}, Trans. Amer. Math.
Soc. {\bf 80} (1955), 387--401.

\bibitem[MFK94]{MFK94} D. Mumford, J. Fogarty and F. Kirwan,
{\em Geometric Invariant Theory}, Third edition, Ergebnisse der Mathematik
und ihrer Grenzgebiete (2), {\bf 34}, Springer-Verlag, Berlin, 1994.

\bibitem[Pr03]{Pr03}  A. Premet,
{\em Nilpotent commuting varieties of reducible Lie algebras}, Invent. Math. {\bf 154} (2003),  653--683.

\bibitem[Ri79]{Ri79}  R. Richardson,
{\em Commuting varieties of semisimple Lie algebras and algebraic groups}, Compositio Math. {\bf 38} (1979), 311--327.

\bibitem[Sch07]{Sch07} A.H.W. Schmitt,
{\em Geometric Invariant Theory Relative to a Base Curve},
In: Pragacz P. (eds) Algebraic Cycles, Sheaves, Shtukas, and Moduli, 2007,
Trends in Mathematics. Birkh\"{a}user Basel.
\end{thebibliography}
\end{document}